\documentclass[12pt,a4paper]{article}
\usepackage{amsmath,amssymb,amsthm,graphicx,color,bbm,array}
\usepackage{soul}
\usepackage[left=1in,right=1in,top=1in,bottom=1in]{geometry}
\usepackage{cite}
\usepackage[british]{babel}
\usepackage{hyperref} 
\usepackage{enumitem}
\usepackage{tikz}
\usepackage{xcolor}
\usetikzlibrary{decorations.markings}
\usepackage{comment}
\hypersetup{
    colorlinks=false,
    pdfborder={0 0 0},
}

\makeatletter 
\@mparswitchfalse%
\makeatother
\reversemarginpar
\usepackage[colorinlistoftodos]{todonotes}

\newtheorem{theorem}{Theorem}[section]

\newtheorem{proposition}[theorem]{Proposition}
\newtheorem{lemma}[theorem]{Lemma}

\numberwithin{equation}{section}

\theoremstyle{definition}

\newtheorem{problem}[theorem]{Problem}

\theoremstyle{remark}
\newtheorem{remark}[theorem]{Remark}
\newtheorem{remarks}[theorem]{Remarks}
\newtheorem*{remark*}{Remark}

 \allowdisplaybreaks

\newcommand{\1}[1]{{\mathbbm{1}\mkern -1.5mu}{\{#1\}}}

\newcommand{\2}[1]{{\mathbbm{1}}_{#1}}

\newcommand{\R}{{\mathbb R}}
\newcommand{\Cl}{\mathfrak{C}}
\newcommand{\Z}{{\mathbb Z}}
\newcommand{\N}{{\mathbb N}}
\newcommand{\ZP}{{\mathbb Z}_+}
\newcommand{\RP}{{\mathbb R}_+}

\newcommand{\eps}{\varepsilon}

\DeclareMathOperator{\Exp}{\mathbb{E}}
\let\Pr\relax
\DeclareMathOperator{\Pr}{\mathbb{P}}

\DeclareMathOperator{\bE}{{\mathbf{E}}}
\DeclareMathOperator{\bP}{{\mathbf{P}}}

\newcommand{\re}{{\mathrm{e}}}
\newcommand{\rc}{{\mathrm{c}}}
\newcommand{\ud}{{\mathrm d}}

\newcommand{\cF}{{\mathcal F}}
\newcommand{\cG}{{\mathcal G}}

\newcommand{\as}{\ \text{a.s.}}

\newcommand{\Bin}{\mathop{\text{Bin}}}

\newcommand{\bigmid}{\; \bigl| \;}

\newcommand{\uq}{{\underline{q}_\alpha}}

\newcommand{\eqd}{\overset{d}{=}}

\newcommand{\teta}{\widetilde{\eta}}

\makeatletter
\def\namedlabel#1#2{\begingroup  
    (#2)%
    \def\@currentlabel{#2}%
    \phantomsection\label{#1}\endgroup
}
\makeatother

\newlist{myenumi}{enumerate}{10}
\setlist[myenumi]{leftmargin=0pt, labelindent=\parindent, listparindent=\parindent, labelwidth=0pt, itemindent=!, itemsep=1pt, parsep=4pt}

\newlist{thmenumi}{enumerate}{10}
\setlist[thmenumi]{leftmargin=0pt, labelindent=\parindent, listparindent=\parindent, labelwidth=0pt, itemindent=!}

\title{Long-range one-dimensional\\ 
internal diffusion-limited aggregation}
\author{Conrado da~Costa\footnote{{Department of Mathematical Sciences, Durham
  University,  Durham DH1 3LE, UK.}}~\footnote{\href{mailto:conrado.da-costa@durham.ac.uk}{conrado.da-costa@durham.ac.uk}} \and Debleena Thacker\footnotemark[1]~\footnote{\href{mailto:debleena.thacker@durham.ac.uk}{debleena.thacker@durham.ac.uk}} \and Andrew
  Wade\footnotemark[1]~\footnote{\href{mailto:andrew.wade@durham.ac.uk}{andrew.wade@durham.ac.uk}}}
\date{\today}

\begin{document}
\maketitle

\begin{abstract}
 We study internal diffusion limited aggregation on $\Z$, where a  cluster
is grown incrementally by adding, for each random walk dispatched from
the origin, the first site it reaches outside the cluster.
We assume that the increment distribution $X$ of the driving random walks has $\Exp X =0$, but need  neither be simple nor symmetric, and can have $\Exp (X^2) = \infty$, for example.
For the case where $\Exp (X^2) < \infty$, we prove that
after $m$ of the random walks have been dispatched, all but $o(m)$ sites in the cluster form  an approximately symmetric contiguous block around the origin. 
This strengthens a result of Blach\`ere, for centred random walks whose increments have
{finite $3$rd moments}, to the optimal moments condition. 
 On the other hand, if~$X$ is
in the domain of attraction of a symmetric $\alpha$-stable law, $1 < \alpha <2$, we prove that
the cluster contains a contiguous block of $\delta m +o(m)$ sites, where $0 < \delta < 1$, but, unlike the finite-variance case, one may not take $\delta=1$.
\end{abstract}

\medskip


\medskip


\tableofcontents 

\section{Introduction and main results}
\label{sec:introduction-and-results}

\subsection{Diffusion-generated growth}
\label{sec:introduction}

Internal diffusion-limited aggregation (IDLA)
is a discrete-time stochastic growth model 
on an infinite graph
driven by a sequence of random walks,
dispatched one after another from a common site (the germ of the aggregate), 
with each walker 
expanding the aggregate by the addition of the first
site the walker visits outside the current aggregate.
The model was
introduced
independently by 
Meakin \& Deutch~\cite{md} and 
by Diaconis \& Fulton~\cite{df} 
in the case where the driving random walk is simple symmetric random walk (SSRW) on $\Z^d$.

 Lawler, Bramson \& Griffeath~\cite{lbg} 
 proved that the long-time shape of the aggregate generated by SSRW
 on $\Z^d$
 converges to a ball.
More recently, after important early work of Lawler~\cite{lawler},
a rich picture concerning sub-diffusive fluctuations around the limit shape for $d \geq 2$ has been revealed~\cite{ag1,ag2,jls1,jls2,jls3}
(it was already observed in~\cite[pp.~107--8]{df} and~\cite[p.~2118]{lbg} that
in $d=1$ the fluctuations of the process can be described via Friedman's urn).

The terminology ``internal'' refers to the fact that the successive walkers are dispatched from inside the current cluster, in contrast to classical DLA~\cite{ws}, in which they are dispatched ``from infinity''. We refer the readers to \cite{Sava-Huss2021} for a survey and comparison of the two different DLA models. One major
difference between these two models is that the IDLA models exhibit regularity in their asymptotic behaviour, for example in $\Z^d, \, d \ge 2$ the limit shape is a Euclidean ball; whereas the shapes generated by classical DLA models appear to be highly irregular and display fractal structure, although mathematical results are scarce (see \cite{Kesten1987,Sava-Huss2021,ws}). \emph{Long-range} classical DLA models have been studied in~\cite{bpp,amir1,amir2,amir3}, driven by random walks that are allowed jumps beyond nearest-neighbours.

The majority of work on IDLA has been concerned with SSRW as the driving random walk;
a variation in which  only a random subset of sites are trapping
was studied in~\cite{ben-arous},
drifted simple random walk was studied by~\cite{lucas}, and excited random walk by~\cite{rs},
while the case of walkers starting uniformly in the present cluster, rather than from a fixed origin,
was studied in~\cite{bdckl}. Other recent work deals with IDLA generated by branching random walk~\cite{ast}. 

Moving away from $\Z^d$, IDLA has also been studied on
cylinder graphs~\cite{silvestri}, on fractals~\cite{chsht,fhksh}, and on regular trees~\cite{bpp}, the latter setting being related to digital search trees~\cite{DrMota_2020}.
A continuous-time version of IDLA was introduced in~\cite{Be_Ka_Pr_2014}, driven by an oriented simple random walk on the upper half plane of $\Z^2$, and admitting a coupling with some first-passage percolation models. Further connections of IDLA driven by simple random walks on $\Z^2$ and random forests have been investigated in~\cite{Che_Cou_Rou_2023}.  

On $\Z^d$, the only work of which we are aware that considers  long-range
random walks is the  paper of Blach\`ere~\cite{blachere}, in which the shape theorem and sub-diffusive fluctuation results of~\cite{lbg,lawler} are extended to walks whose increments
have zero mean and {finite $3$rd moments}. Part of the contribution of the present paper (as we discuss in more detail shortly) is to extend this result by reducing the required moments for the shape theorem to finite variance only, which is best possible. We also consider the case of infinite variance, where we know of no previous quantitative results.

In the present paper, 
we investigate the degree to which the regularity exhibited by classical (SSRW-driven) IDLA is preserved
for long-range IDLA on $\Z$.
We will consider 
random walks whose increments have mean zero, and either finite variance
(where we provide essentially-optimal extensions of results of~\cite{blachere} by removing higher-order moments assumptions), or are in the domain of normal attraction of a symmetric $\alpha$-stable distribution with characteristic function ${\re^{-\beta |t|^\alpha}}$
for $1 < \alpha <2$ (for which we know of no previous results in the literature). 
Our main results are in two parts, and are presented in detail
in \S\ref{sec:results} below, after we have introduced necessary notation in \S\ref{sec:model}). 
First, we show that the long time behaviour observed for SSRW-driven IDLA extends  to the mean-zero, finite-variance case (Theorem~\ref{thm:idla-light-tail}). {This improves on the best available previous result, due to Blach\`ere~\cite{blachere}, which, as stated, demands finiteness of $4$th moments\footnote{{Blach\`ere's result also provides quantitative bounds on fluctuations, which we do not pursue here.}}, although an inspection of the arguments of~\cite{blachere} suggests that $3$rd  moments is enough to make it work (see Remark~\ref{rems:SSRW}\ref{rems:SSRW-ii} below)}.
Second, we show that the regular behaviour
begins to degrade in the infinite-variance case (Theorem~\ref{thm:idla-heavy-tail}), demonstrating a phase transition in the model, and quantifying that the finite-variance result is optimal. {We give some intuition (in terms of ideas from renewal theory) for why finite vs.~infinite variance is the critical point for this phase transition in \S\ref{sec:discussion} below.}

The complementary work \cite{Balasz_etal2024} introduced a ``one-sided'' IDLA model on $\Z$ where the initial semi-infinite 
aggregate consists of sites $\{0,-1,-2,\ldots\}$ (say) and each walker comes from $-\infty$ (as opposed to our model where the walkers start from the origin), and steps on the renewal points of a bi-infinite (potentially long-range) renewal chain  and settles at the first vacant site it encounters. The main object of interest in \cite{Balasz_etal2024} is the 
evolution of the ``shock profile'' of the moving front of the cluster (holes to left, islands to the right), including
 a connection to  blocking measures arising from the asymmetric simple exclusion process.

\subsection{Model and notation}
\label{sec:model}

Consider a probability space $(\Omega, \cF, \Pr)$
which supports an array of i.i.d.~$\Z$-valued random variables~$X$ and~$X^{(m)}_i$, $m \in \N$, $i \in \N$. (Throughout the paper, we write $\N := \{1,2,3,\ldots\}$
and $\ZP := \{ 0 \} \cup \N$.) 
Consider  $S^{(1)}, S^{(2)}, \ldots$ a sequence of independent random walks,
started from the origin, defined via $S^{(m)}_0 :=0$ and
\[ S^{(m)}_n := \sum_{i=1}^n X^{(m)}_i, \text{ for } n \in \N. \]
For convenience, we also consider a random walk whose increment distribution
is the same as that of $X$ above, but with an arbitrary starting point;
to this end we introduce (for each $x \in \Z$) a probability space $(\Omega', \cF', \bP_x)$
which supports 
a random walk $S = ( S_0, S_1, S_2, \ldots)$
such that $\bP_x (S_0 = x) = 1$, and $S_{n+1} - S_n$ are i.i.d.~with
$\bP_x (S_{n+1} - S_n = y) = \Pr ( X = y)$ for all $x, y \in \Z$.
We write $\Exp, \bE_x$ for expectation corresponding to~$\Pr, \bP_x$, respectively.

Throughout this paper
we assume the following \emph{irreducibility} hypothesis, which ensures that the random walk
can visit all of $\Z$:
\begin{description}
\item[\namedlabel{ass:irreducible}{I}] 
Suppose that 
for every  $x,y \in \Z$,  there is $n \in \N$ such that
\begin{equation}
\label{hyp:irreducible}
    \bP_x (S_n = y)>0.
\end{equation}
\end{description}

\begin{remark}
    \label{rem:irreducible}
An elementary sufficient condition for~\eqref{ass:irreducible},
which holds for simple symmetric random walk and many other examples, is that
(i) $X$ is not constant and has $\Exp X = 0$ (\emph{zero mean drift}), and (ii) 
for no $h>1$ does it hold that 
$\Pr (X \in h \Z ) =1$; this last condition is no real loss of generality, as if $\Pr (X \in h \Z ) =1$ for $h >1$,
then one can work instead with increment $X/h$. To see the sufficiency, note that if (i) holds, then the support of~$X$ contains at least one positive value $x_+ \in \N$ and at least one negative value $-x_-$ for $x_- \in \N$, while if (ii) holds, $x_+, x_-$ can be chosen to have $\gcd(x_+,x_-) = 1$.
\end{remark}

Define $\Cl_0 := \{ 0 \}$ (the germ) and then, recursively,
 for $m \in \N$ let 
\begin{equation}\label{eq:cluster expansion}
 \tau_m := \inf \bigl\{ n \in \ZP : S^{(m)}_n \notin \Cl_{m-1} \bigr\} ,
  ~\text{and}~ \Cl_m := \Cl_{m-1} \cup \{ S^{(m)}_{\tau_m} \} .
\end{equation}
Hypothesis~\eqref{ass:irreducible} implies that $\limsup_{n \to \infty} | S^{(m)}_n | = \infty$, $\Pr$-a.s., and hence $\tau_m < \infty$, $\Pr$-a.s., for every $m \in \N$. Consequently, the number
of sites in $\Cl_m$ is $\# \Cl_m = m+1$.

We make two comments on the notation. First, we use~$n$ for the internal clock for each walker, and keep $m$ for indexing the walkers. Second, while $S^{(m)}$ denotes the $m$th random walker to be released in our IDLA
process, we use $S$ and $\bP_x$~to make statements generically about the random walk with increments distributed as~$X$.

We call $\Cl_m$ the \emph{IDLA cluster} generated by the first $m$ walkers.
By construction, $\Cl_{m} \subset \Cl_{m+1}$ are (strictly) increasing;
denote the limit by $\Cl_\infty := \cup_{m \in \ZP} \Cl_m$, the collection of sites eventually contained in the cluster. The limit set $\Cl_\infty$ is an infinite subset of $\Z$; the following elementary result,
whose proof is in Appendix~\ref{sec:appendix},
says that it is, in fact, the whole of $\Z$, with no gaps.

\begin{proposition}
\label{prop:fill-finite-set}
Suppose that~\eqref{ass:irreducible} holds.
Then $\Pr ( \Cl_\infty = \Z ) =1$.
\end{proposition}

To describe our main results define, for $m \in \ZP$, 
\begin{equation}
\label{eq:max-radius}
r_m :=  \max \{ r \in \ZP :  \Z \cap [-r,r] \subseteq \Cl_m \} , 
\end{equation}
 the radius of the maximal centred interval contained in $\Cl_m$.
An easy consequence of the fact that $\# \Cl_m = m+1$ is that
\begin{equation}
    \label{eq:r-R-trivial}
    r_m \leq m/2, \text{ for all } m \in \ZP. 
\end{equation}

\subsection{Main results}
\label{sec:results}

It follows from Proposition~\ref{prop:fill-finite-set} that
$\lim_{m \to \infty} r_m = \infty$, a.s.
We are interested in quantifying the growth rate of~$r_m$.
Loosely speaking, our first result, Theorem~\ref{thm:idla-light-tail}, says that when $\Exp (X^2) < \infty$ and $\Exp X =0$, the inner radius $r_m$ grows at the maximal rate permitted by~\eqref{eq:r-R-trivial}. Here is the precise statement; we give some intuition for this result, and the criticality of the second-moment condition, in \S\ref{sec:discussion} below.

\begin{theorem}
\label{thm:idla-light-tail}
Suppose that~\eqref{ass:irreducible} holds, $\Exp ( X^2) < \infty$, and $\Exp X=0$. Then, a.s.,
\begin{equation}
\label{eq:inner_radius-lim}
    \lim_{m \to \infty} \frac{r_m}{m} = \frac{1}{2}.
\end{equation}
\end{theorem}
\begin{remarks}
\phantomsection
\label{rems:SSRW}
\begin{myenumi}[label=(\roman*)]
\item\label{rems:SSRW-i}
For SSRW, when $\Pr (X = +1 ) = \Pr (X = -1) =1/2$,
it is well known that $\lim_{m \to \infty} r_m /m  = 1/2$, a.s. Indeed, in this case $\Cl_m$ is always a contiguous interval, and 
the equivalence of the model to (Bernard) Friedman's urn, as described at~\cite[pp.~107--8]{df} and~\cite[p.~2118]{lbg}, yields $\lim_{m \to \infty} r_m/m = 1/2$, a.s., via (David) Freedman's strong law~\cite{freedman}. 
\item\label{rems:SSRW-ii}
{The $d=1$ case of the  ``$\liminf$'' half of Theorem~2.3 of~\cite{blachere} yields~\eqref{eq:inner_radius-lim} for~$X$ with $\Exp X =0$ and $\Exp ( | X|^{p} )<\infty$ for~$p = 4$. An inspection of the proof of the $d=1$ case of Theorem~2.3 of~\cite{blachere} (in \S2.1 of the paper) shows that the $p=4$ hypothesis could, we believe, be relaxed to $p=3$ without changing the proofs.
In Lemma 2.9 (which is specific to $d=1$), the hypothesis of finite $3$rd moments is used to apply an estimate for exiting a symmetric interval on the right or left with a quantitative error bound;
Kesten's gambler's ruin estimate,
that we describe in our Lemma~\ref{lem:kesten-gambler} below, needs only $p=2$ but does not provide the quantitative error term that is apparently needed to make the argument in~\cite{blachere} work. 
 The $p=4$ hypothesis is used in Lemma 2.6 in~\cite{blachere} to apply the $q=2$ case of Lemma 2.5 there, which needs $p=2+q$ moments, and provides bounds on moments of the exit point from the cluster. However, the full $p=4$ assumption appears needed only for $d \geq 2$. For $d=1$  the argument needs 
 (in Lemma~2.6 at the top of p.~623 of~\cite{blachere}
 and again in Lemma~2.9 at the bottom of p.~626)
 the $q=1$ case of Lemma 2.5, and hence $p = 3$ moments. This is related to
  the renewal-theoretic fact that the overshoot of a distant level by zero-mean random walk has a uniformly finite mean
 only if the original increment has a finite 3rd moment, as discussed in~\S\ref{sec:overshoots} below; Lemmas 2.4--2.5 of Blach\`ere~\cite{blachere}
 correspond to versions of this phenomenon in general dimensions.
 Thus $p=3$ seems to be an essential
 requirement for the argument in~\cite{blachere} as presented there. 
  The main contribution of our Theorem~\ref{thm:idla-light-tail} is to reduce the hypothesis on moments required for~\eqref{eq:inner_radius-lim}  to finiteness of variance ($p=2$). The  origin of the higher moments requirements in~\cite{blachere} seems to be the approach there to bounding Green's functions of the walk: e.g., in Lemma 2.4 of~\cite{blachere}  the value of a Green's function in an annulus is bounded by the square of the radius. Our approach (which in this respect follows~\cite{lbg}) uses binomial concentration over a relatively large number of walkers to obtain high probability estimates, rather than bounding the corresponding single-walker moments, as we explain in \S\ref{sec:inner-radius-heuristics} below.}
 \item\label{rems:SSRW-iii}
 {  Theorem~2.3 of~\cite{blachere}
 also provides 
(under further moments conditions, in some cases, with at least $p>7$)
 upper bounds on  fluctuations; e.g., when $X$ has all moments, for every $\eps>0$ it holds that $| r_m - m | = o ( m^{(1/2)+\eps} )$, a.s.}
 \item\label{rems:SSRW-iv}
{ Blach\`ere's 
 Proposition~2.20~\cite{blachere}
 shows that the $p=2$ moments hypothesis in Theorem~\ref{thm:idla-light-tail} is best possible;
 our 
 Theorem~\ref{thm:idla-heavy-tail} below gives, as far as we know,
 the first  quantitative result in the infinite-variance setting.}
 \end{myenumi}
\end{remarks}

We are also interested in the case where $\Exp (X^2) = \infty$. We will explore this case
under the following stable-domain hypothesis for $1 < \alpha < 2$.
\begin{description}
\item
[\namedlabel{ass:stable-alpha}{S$_\alpha$}]
Suppose that
$X \in \Z$ has a symmetric distribution (i.e., $X \eqd - X$), and
its characteristic function $\phi (t) := \Exp ( \re^{itX} )$
satisfies
\begin{equation}
    \label{eq:kesten-assumption}
    \lim_{t \to 0 } \left( | t|^{-\alpha} ( 1 - \phi (t) ) \right) = \beta \in (0,\infty).
\end{equation}
\end{description}
\begin{remarks}
\phantomsection
\label{rems:stable}
\begin{myenumi}[label=(\roman*)]
\item\label{rems:stable-i}
The hypothesis~\eqref{eq:kesten-assumption} is equivalent
to the assumption that $X$ is in the domain of normal attraction of~$\zeta_\alpha$, the symmetric $\alpha$-stable distribution with characteristic function {$\re^{-\beta |t|^\alpha}$}: see Theorem 2.6.7~\cite[pp.~92--3]{il}. Recall that being in the domain of normal attraction $\zeta_\alpha$ means that $n^{-1/\alpha} S_n$ converges in distribution to $\zeta_\alpha$, i.e., the slowly-varying component of the scaling sequence is constant.   
\item\label{rems:stable-ii}
For $\alpha \in (1,2)$, assumption~\eqref{ass:stable-alpha} implies that $\Exp |X| < \infty$ (in fact $\Exp ( |X|^\gamma ) < \infty$ for every $\gamma < \alpha$), with $\Exp X = 0$ (due to symmetry), 
but $\Exp ( X^2 ) = \infty$~\cite[p.~93]{il}.
\end{myenumi}
\end{remarks}

Loosely speaking, our second main result says that, under hypothesis~\eqref{ass:stable-alpha}
with $\alpha \in (1,2)$, it still holds that $r_m$ grows at linear rate with $m$ (in $\liminf$ and $\limsup$ sense),
but this rate is now strictly less than the maximal rate $1/2$ permitted by~\eqref{eq:r-R-trivial}. That  $r_m \not\to 1/2$ in the case where $\Exp (X^2 ) = \infty$ is already implied by Proposition~2.20 of~\cite{blachere}; the next result gives more quantitative information.
We give some intuition for the contrast between Theorems~\ref{thm:idla-light-tail}  and~\ref{thm:idla-heavy-tail} in \S\ref{sec:discussion} below.

\begin{theorem}
\label{thm:idla-heavy-tail}
Suppose that~\eqref{ass:irreducible} holds, 
and that~\eqref{ass:stable-alpha} holds with $\alpha \in (1,2)$. Then, there exist
constants $c_\alpha, c'_\alpha$
with $0 < c_\alpha \leq c'_\alpha < 1/2$, such that, 
\begin{equation}
    \label{eq:r-m-heavy}
 c_\alpha \leq  \liminf_{m \to \infty} \frac{r_m}{m} \leq \limsup_{m \to \infty} \frac{r_m}{m} \leq c'_\alpha, 
  \as
\end{equation}
Moreover, one may take for the constant $c_\alpha$ in~\eqref{eq:r-m-heavy} the expression
\begin{align}
\label{eq:c-alpha}
    c_\alpha & =  \frac{(\alpha -1)  (2-\alpha)^{2-\alpha}}{(4-\alpha)(3-\alpha)^3} .
\end{align}
\end{theorem}

\begin{remark}
\label{rem:idla-heavy-tail}
The value of $c_\alpha$ in~\eqref{eq:c-alpha} is not the best that
can be extracted from our method (see \S\ref{sec:inner-radius-infinite-variance-lower-bound}), but
was chosen for its relatively simple formula, together with its property that $c_\alpha \uparrow 1/2$ as $\alpha \uparrow 2$, which demonstrates some continuity between Theorem~\ref{thm:idla-heavy-tail} and Theorem~\ref{thm:idla-light-tail}.
We do not give here an explicit expression for $c_\alpha'$, 
though an explicit but not very informative bound could be extracted using our methods.
\end{remark}

\subsection{Overview and discussion}
\label{sec:discussion}

{Before giving the outline of the paper and stating some open problems, we give some intuition behind the main results,  Theorems~\ref{thm:idla-light-tail} and~\ref{thm:idla-heavy-tail}, which also indicates
some of the mathematical ideas that are needed in the formal proofs. Suppose that the current cluster consists of the completely filled interval $[-x,x]$ and some sporadic sites outside this. To understand the growth of the cluster, consider a random walker that starts from the origin and is stopped when it exits the interval $[-x,x]$. Under our hypotheses (either the increment distribution has finite 
variance and zero mean, or has infinite variance and is symmetric) the walk has probability close to $1/2$ of exiting on the right (say) and how far it travels from the cluster is related to the \emph{overshoot} of level~$x$. It follows from classical random walk fluctuation and renewal theory
that the overshoot is tight (as $x \to \infty$) if $\Exp ( X^2) < \infty$, but, under hypothesis~\eqref{ass:stable-alpha} for $\alpha \in (1,2)$, the overshoot itself lives on scale~$x$ (we collect relevant renewal-theoretic tools in \S\ref{sec:ingredients} below). In the finite-variance case, a walker exiting the cluster $[-x,x]$ over $x$ thus lands within distance $O(1)$ and will, with high probability, visit any site in the neighbourhood of $x$ before reaching $u x$ ($u \approx 1$ but $u>1$). This gives some intuition behind Theorem~\ref{thm:idla-light-tail}. On the other hand, in the infinite-variance case, a walker exiting the cluster $[-x,x]$ over $x$ will, with appreciable probability, land to the right of $Ax$ for large $A > 1$, and thus not contribute to growing the bulk of the cluster until much later, which gives  intuition behind Theorem~\ref{thm:idla-heavy-tail}.
}

The bulk of the rest of the paper provides the 
proofs of Theorems~\ref{thm:idla-light-tail} and~\ref{thm:idla-heavy-tail}.
Section~\ref{sec:ingredients} presents preparatory results concerning properties of the underlying random walk; 
these bring together known results from random walk and renewal theory (\S\ref{sec:overshoots}), with some 
important hitting and exit estimates for integer-valued random walks (\S\ref{sec:kesten}), due to Kesten~\cite{kesten1961a,kesten1961b}. The proofs of the main results are separated into proofs of lower bounds on the growth rate of $r_m/m$ (\S\ref{sec:inner-radius}) and upper bounds on the growth rate  of $r_m/m$ in the infinite variance case  (\S\ref{sec:inner-radius-infinite-variance-lower-bound}).
The strategy of \S\ref{sec:inner-radius}
adapts, in part, the approach of~\cite{lbg} (as did~\cite{blachere}), combined with the results of Kesten mentioned above;
an outline of the argument is presented in \S\ref{sec:inner-radius-heuristics}. 
Section~\ref{sec:inner-radius-infinite-variance-lower-bound} uses further results of Kesten, in the neighbourhood of the Dynkin--Lamperti renewal theorem (see \S\ref{sec:overshoots}). 
The proof of Theorem~\ref{thm:idla-light-tail} is accomplished in \S\ref{sec:inner-radius-finite-variance}, while the proof of Theorem~\ref{thm:idla-heavy-tail} combines a lower bound from \S\ref{sec:inner-radius-infinite-variance-upper-bound} with upper bound from \S\ref{sec:inner-radius-infinite-variance-lower-bound}, and is concluded in the latter section. 
Finally, to avoid disrupting the flow of the paper, we defer to
the appendix the proofs of some auxiliary results. In particular, 
Appendix~\ref{sec:appendix} gives the short
proof of the eventual filling statement in Proposition~\ref{prop:fill-finite-set}, and
Appendix~\ref{sec:appendix-equicontinuity} some technical elements about
certain families of probability functions that are introduced in \S\ref{sec:kesten}.

We finish this section with some remarks and open problems. 
Theorem~\ref{thm:idla-heavy-tail} poses some obvious questions. Firstly:

\begin{problem}
Suppose that~\eqref{ass:irreducible} holds, 
and that~\eqref{ass:stable-alpha} holds with $\alpha \in (1,2)$. Does it hold that $\liminf_{m \to \infty} r_m/m = \ell_\alpha$, a.s., for some constant $\ell_\alpha$, a.s.? If so, what is its value? 
\end{problem}

Of course, if it exists, $\ell_\alpha$ must satisfy $c_\alpha \leq \ell_\alpha \leq c'_\alpha$, by~\eqref{eq:r-m-heavy},
and (see Remark~\ref{rem:idla-heavy-tail}) it must hold that $\ell_\alpha \to 1/2$ as $\alpha \uparrow 2$. One might hope to be able to apply a zero--one law to obtain existence of $\ell_\alpha$, but we have not been able to do so.  Two further questions in this regime are the following.

\begin{problem}
\label{prob:limsup-r}
Suppose that~\eqref{ass:irreducible} holds, 
and that~\eqref{ass:stable-alpha} holds with $\alpha \in (1,2)$. Does  $\lim_{m \to \infty} r_m/m$ exist in this case,
as it does in Theorem~\ref{thm:idla-light-tail}?  
\end{problem}

Further questions arise in the case with $\Exp |X| = \infty$.

\begin{problem}
Suppose that~\eqref{ass:stable-alpha} holds with $\alpha \in (0,1)$, or $\alpha =1$. What is the behaviour of $r_m$ now?
\end{problem}

The underlying random walks are of very different character in the case $\alpha \in (0,1]$: for example,
when $\alpha \in (0,1)$ the walks are (oscillatory) transient; this does not obviously indicate
a more disperse aggregate, however, as the walks will typically aggregate after 
fewer steps. 
The boundary case $\alpha =1$ is likely to be delicate, and there are some technical obstructions: for example, the results of Kesten~\cite{kesten1961a,kesten1961b} that we use below often omit the case $\alpha =1$, although comparable results for stable diffusions are known~\cite{Ch_Ki_So2010}.
Lastly:

\begin{problem}
\label{prob:higher-dimensions}
Consider random walks in $\Z^d$, $d \geq 2$.
\end{problem}

 The seminal work~\cite{lbg} provides a shape theorem (convergence to a ball) for 
SSRW on $\Z^d$, and Theorem 2.3~\cite{blachere}
extends the shape theorem (convergence to a ball, up to linear transformation to account for covariance) to  walks with mean zero and sufficiently many moments ($11+4d+\varepsilon$ moments suffice for all cases, but for some results~$4$ moments is enough).
One aspect of Problem~\ref{prob:higher-dimensions}
would be to   extend the shape theorem to walks with  finite $(d+1)$th moments, the necessary condition 
provided by Proposition 2.20 of~\cite{blachere}. The infinite-variance (indeed, infinite $(d+1)$th-moment) case seems largely open for $d \geq 2$. 
For stable-domain walks of index $\alpha \in (1,2)$,
we know of no quantitative results for multidimensional IDLA, and also no multidimensional analogues of Kesten's results that we apply for the $d=1$ case studied in the present paper,
although for symmetric stable processes
some Green's function estimates are known~\cite{Ch_Ki_So2010}.

\section{Ingredients from random walks and renewal theory}
\label{sec:ingredients}

\subsection{Ladder processes and overshoots}
\label{sec:overshoots}

Recall that on probability space $(\Omega', \cF', \bP_x)$, $x \in \Z$,
we have a random walk $S = (S_n)_{n \in \ZP}$ started from $S_0 =x$
whose increment distribution is that of the $\Z$-valued random variable~$X$ underlying our IDLA model.
We introduce some additional notation to enable us to discuss classical fluctuation theory
for this random walk. 
Define the strict ascending ladder times $\lambda_0:=0$ and
\begin{equation}
\label{eq:ascending-ladder-times}
\lambda_k := \inf \{ n \geq \lambda_{k-1} : S_n > S_{\lambda_{k-1}} \}, \text{ for } k \in \N;     
\end{equation}
as usual, $\inf \emptyset := \infty$. 

We suppose that $\Exp |X| < \infty$ and $\Exp X =0$. By a result of Chung and Fuchs,
this implies that the random walk $S_n$ is recurrent~\cite[p.~615]{feller2}, so,
for every $x \in \Z$,
\begin{equation}
\label{eq:limsup}
\bP_x \Bigl( \limsup_{n \to \infty} S_n = \infty \Bigr) = \bP_x \Bigl( \liminf_{n \to \infty} S_n = -\infty \Bigr) = 1 .
\end{equation}
In particular, $\bP_x (\lambda_k < \infty ) = 1$ for every $k \in \ZP$, and
$S_{\lambda_k} = \max_{0 \leq n \leq \lambda_k} S_n$ are strictly increasing in~$k$.

Define $L_n := S_{\lambda_n}$ for $n \in \ZP$; then
$x = L_0 < L_1 < L_2 < \cdots$ is the (strict, ascending) \emph{ladder height} process associated with $S$.
By the strong Markov property and spatial homogeneity, the increments $Y_n := L_n - L_{n-1}$, $n \in \N$,
are i.i.d., $\N$-valued, and $L_n = x+\sum_{i=1}^n Y_i$ represents the ladder height process as a renewal process
with increment distribution~$Y := Y_1$. 
By translation invariance, the distribution of the ladder times~$\lambda_k$ and ladder increments~$Y_n$
are the same for every $\bP_x$, regardless of the starting point $x \in \Z$ of the random walk. In such cases where the starting point is unimportant, we will abuse notation slightly and write simply $\bP$, $\bE$ on occasion.

Suppose next that $S_0 = x \equiv 0$.
For $y \in \ZP$, define the renewal
counting process
\[ N_y := \inf \{ n \in \N : L_n > y\} = \# \{ n \in \ZP : L_n \leq y \} .\]
Note $L_{N_0} = L_1 = Y_1$, $N_y \in \N$, and $L_{N_y -1} \leq y < L_{N_y}$, $\bP_0$-a.s.~for every~$y \in \ZP$.
Define the \emph{residual life-time} process associated with the ladder-height renewal process by
\begin{equation}
\label{eq:def_residual}
    Z_y := L_{N_y} - y, \text{ for all } y \in \ZP,
\end{equation}
which satisfies $Z_y \in \N$; see~\cite[pp.~140--1]{asmussen} or~\cite[\S XI.4]{feller2}
for background on the terminology and renewal-theoretic context. A fundamental observation~\cite[p.~9]{asmussen} is that
 $Z_0, Z_1, Z_2, \ldots$ forms an irreducible  Markov chain on 
 a
 subset of $\N$ with transitions given by
 \begin{equation}
     \label{eq:overshoot-chain}
 Z_{y+1} = \begin{cases} Z_y -1 & \text{if } Z_y \geq 2, \\ Y_{N_{y+1}} &\text{if } Z_y = 1 .\end{cases}
 \end{equation}
Furthermore, the Markov chain $(Z_n)_{n \in \ZP}$ is aperiodic, i.e. $\text{gcd} \{n \in \ZP: Z_n = 1\} = 1$.
Indeed (see Remark~\ref{rem:irreducible}) in this case the support of $X$ contains some $x_+$ and $-x_-$,
with $x_+, x_- \in \N$ and $\gcd (x_+, x_-) =1$. Hence we can find $k, \ell \in \N$ with $k x_+ - \ell x_- = 1$, and with positive probability the random walk $S$ can take $\ell$ steps of value $x_-$ followed by $k$ steps of value $x_+$, meaning that the ladder variable has $\Pr (Y=1) >0$.

Returning to the random walk, we denote the first passage time of~$S$ above level $y \in \ZP$ by 
\begin{equation}
\label{eq:rho_x-def}
    \rho_y := \inf \{ n \in \ZP : S_n > y \} ,
\end{equation}
and we call $S_{\rho_y} - y$ the first (right) \emph{overshoot} of level $y$ by the random walk~$S$. 
Since $\rho_y$ is necessarily a ladder time, it holds that
\begin{equation}
\label{eq:overshoot-residual}
S_{\rho_y} - y = L_{N_y} - y = Z_y, ~\bP_0\text{-a.s.}, \text{ for every } y \in \ZP; \end{equation}
thus overshoots of the random walk are equivalent to residual life-times of the associated ladder-height renewal process.

\begin{remark}
\label{rem:left-overshoots}
We state all the results in this section for \emph{right} overshoots and
\emph{increasing} ladder variables, but, evidently, by working with the increment
distribution~$-X$, we can translate everything to left overshoots and decreasing ladder variables.
\end{remark}

For most of the rest of this section, we assume additionally that the increments have finite variance:
\begin{description}
\item
[\namedlabel{ass:moments}{M}]
Suppose that $\sigma^2 := \Exp ( X^2 ) \in (0,\infty)$ and $\Exp X = 0$.
\end{description}
The following result presents some key properties of overshoots.

\begin{proposition}
\label{prop:walk-overshoot}
Suppose that~\eqref{ass:irreducible} and~\eqref{ass:moments} hold. 
Then $\mu :=  \bE Y$ has $1 \leq \mu < \infty$, and
\begin{equation}
    \label{eq:mu-formula}
   \mu = \exp \left\{ \sum_{n=1}^\infty \frac{1}{n} \left[ \frac{1}{2} - \bP ( S_n > 0 ) \right] \right\} .
\end{equation}
Moreover, define the probability distributions $(\pi_k)_{k \in \N}$ 
and $(\psi_k)_{k \in \N}$ by
\begin{equation}
    \label{eq:pi-k}
 \pi_k := \frac{\bP ( Y \geq k )}{\mu} , ~~~
 \psi_k := \frac{ k \bP ( Y = k)}{\mu}, 
 \text{ for } k \in \N.\end{equation}
 \begin{thmenumi}[label=(\roman*)]
    \item
    \label{prop:walk-overshoot-i}
Let $Z_\infty$, $L_\infty$ denote independent
random variables with distributions given by $\pi$, $\psi$ from~\eqref{eq:pi-k}, respectively, and set $U_\infty := \max (Z_\infty, L_\infty)$. Then
\[ \sup_{y \in \ZP} \bP_0 ( Z_y \geq k ) \leq \bP_0 ( U_\infty \geq k) , \text{ for all } k \in \ZP. \] 
    \item
    \label{prop:walk-overshoot-ii}
    It holds that, for every~$x \in \Z$ and every~$k \in \N$, 
\begin{equation}
    \label{eq:pi-limit}
 \lim_{y \to \infty} \bP_x ( Z_y = k  ) = \pi_k .\end{equation}
 \item 
 \label{prop:walk-overshoot-iii}
 Suppose, additionally, that $\Exp ( |X|^p ) < \infty$, for some $p >2$. Then, 
\begin{equation}
\label{eq:overshoot-moments-bound}
\sup_{y \in \ZP} \bE_x \bigl[ Z_y^{p-2} \bigr] < \infty .\end{equation}
\end{thmenumi}
\end{proposition}

In particular, 
Proposition~\ref{prop:walk-overshoot}\ref{prop:walk-overshoot-ii} says that if $\Exp (X^2) < \infty$,
the overshoots are tight. This is in abrupt contrast to the case where $\Exp (X^2) = \infty$,
where, under appropriate conditions, the following consequence of the~\emph{Dynkin--Lamperti theorem} (see Proposition~\ref{prop:dynkin-lamperti})
says that the overshoot over level~$y$ lives on scale~$y$, asymptotically. Note that the result is valid for all $\alpha \in (0,2)$, although we will later use only the case $\alpha \in (1,2)$.

\begin{proposition}
\label{prop:dynkin-lamperti}
Suppose that~\eqref{ass:irreducible} holds, 
and that~\eqref{ass:stable-alpha} holds with $\alpha \in (0,2)$. Then,
for every $x \in \Z$ and every $u \geq 0$,
\[ \lim_{y \to \infty} \bP_x \left( \frac{S_{\rho_y} -y}{y} > u \right) = \int_u^\infty f_\alpha (v) \ud v , \]
where
\[ f_\alpha (v) := \frac{ \sin (\pi \alpha/2)}{\pi} \frac{1}{v^{\alpha/2} (1+v)}, ~\text{ for } v > 0.\]
\end{proposition}

Propositions~\ref{prop:walk-overshoot} and~\ref{prop:dynkin-lamperti}
are well known:  Kesten's Lemma~6~\cite[p.~255]{kesten1961a}
provides Proposition~\ref{prop:dynkin-lamperti} explicitly,
and gives Proposition~\ref{prop:walk-overshoot}\ref{prop:walk-overshoot-ii} under an additional symmetry assumption.
Another route to Proposition~\ref{prop:dynkin-lamperti}
is to combine the Dynkin--Lamperti renewal theorem~\cite[p.~361]{bgt}
with the result that under hypothesis~\eqref{ass:stable-alpha}, the ladder variable~$Y$ is in the domain
of attraction of a positive $\alpha/2$-stable law  (see Theorem~9 of Rogozin~\cite[p.~592]{rogozin});
a corresponding local limit theorem is given in~\cite{doney}. We give below a proof of  Proposition~\ref{prop:walk-overshoot}
without Kesten's additional hypothesis, but the proof involves little more than indicating appropriate results in the literature. 
First, we
give some intuition behind the
important
``loss of   moments'' phenomenon which the above results exhibit. For example,
 Proposition~\ref{prop:walk-overshoot} says that in order for the overshoot to have a uniformly bounded mean, we need to assume $\Exp ( |X|^3 ) < \infty$. 
 
Theorem~3.4 of Spitzer~\cite[p.~158]{spitzer} 
shows that the hypothesis~\eqref{ass:moments} (finite variance)
implies integrability of the ladder height, $1 \leq \bE Y < \infty$,
and gives the formula~\eqref{eq:mu-formula}. We ``lose moments'' in passing from the
walk to its ladder heights (this cannot be avoided, as explained in Remark~\ref{rem:doney} below). The fact that we ``lose another moment'' in passing from the ladder
heights to the (stationary) overshoots is due to the observation that, if  $Z$ is a random variable
distributed as  $\bP (Z = k) = \pi_k$ from~\eqref{eq:pi-k}, then 
 \begin{equation}
 \label{eq:Z-Y-moments}
 \bE ( Z^{q} ) = \sum_{k \in \N} k^{q} \pi_k 
= \frac{1}{\bE Y} \sum_{k \in \N} k^{q} \bP ( Y \geq k ),
\end{equation}
which is finite if and only if $\bE ( Y^{q+1}) < \infty$. This is a ``size-biasing'' effect; in the stationary renewal process associated with~$Y$, the intervals that straddle a particular value are more likely to be long.

\begin{remark}
\label{rem:doney}
Suppose that~\eqref{ass:irreducible} and~\eqref{ass:moments} hold. 
Since $Y \geq 1$ we have $L_n \geq n$ and hence $N_y \leq y+1$, a.s., and  
$L_{N_y} - y \leq L_{y+1} - y$.
Hence, for every $y \in \ZP$, it holds that 
\begin{equation}
    \label{eq:Y-gives-Z}
\bE (Z_y^q ) < \infty, \text{ whenever } \bE (Y^q ) < \infty.
\end{equation}
The purpose of this remark is to explain that we cannot claim that~\eqref{eq:Y-gives-Z} holds uniformly in~$y$. Let $Z$ denote a random variable
whose distribution is given by $\bP (Z = k) = \pi_k$ as given by~\eqref{eq:pi-k};
Proposition~\ref{prop:walk-overshoot}\ref{prop:walk-overshoot-ii}
shows that $Z_y$ converges to $Z$ in distribution as $y \to \infty$.
Suppose that $\sup_y \bE ( Z_y^q) < \infty$. Then uniform integrability
shows that, 
\begin{equation}\label{uniform-bd-integrability}
\text{
for every $q' \in (0,q)$, $\bE (Z^{q'}) = \lim_{y \to \infty} \bE (Z_y^{q'}) < \infty$. }
\end{equation}
In the special case where $X$ is symmetric,
Corollary~2 of Doney~\cite[p.~250]{doney-ladder} states that,
\begin{equation}\label{Doneys_result}
\text{for $p>2$,
$\Exp ( | X |^p ) < \infty$ if and only if $\bE ( Y^{p-1} ) < \infty$.} 
\end{equation}
In particular, if for some $q >1$ one has 
$\Exp ( |X|^{q+1}) < \infty$ but $\Exp (|X|^{q+(3/2)}) =\infty$, say,
then Doney's result \eqref{Doneys_result} says $\bE (Y^q ) < \infty$ but $\bE (Y^{q+(1/2)} ) = \infty$,
and so $\bE (Z^{q-(1/2)} ) = \infty$, by the
discussion around~\eqref{eq:Z-Y-moments}.
This is a contradiction with~\eqref{uniform-bd-integrability}.
Hence
we cannot, in general, insert a supremum over~$y$ into~\eqref{eq:Y-gives-Z}.
\end{remark}

\begin{proof}[Proof of Proposition~\ref{prop:walk-overshoot}]
Suppose, without loss of generality, that $S_0 = 0$, 
and consider $Z_y = S_{\rho_y} - y$ for some $y \in \ZP$.
As mentioned above, the fact that $\mu < \infty$ satisfies~\eqref{eq:mu-formula} is due to Spitzer~\cite{spitzer}. 
Part~\ref{prop:walk-overshoot-i} is an inequality
in the vein of Lorden~\cite{lorden},
obtained by Chang~\cite{chang}, using a coupling argument.

The most elegant (and probabilistic) argument for part~\ref{prop:walk-overshoot-ii}
proceeds from the
Markov chain representation~\eqref{eq:overshoot-chain}.
Indeed,
$Z_0 = 1$ and $Z_1 =Y_{N_1} = Y_1$. If we set
$\tau := \inf \{ n \in \N : Z_n = 1\}$, then $\tau = Y_1$, a.s., and 
the usual excursion-occupation construction shows that an invariant measure $(\mu(y), y \in \N)$ for the Markov chain is given by \[ 
\mu(y) = \frac{1}{\bE \tau} \bE \sum_{n = 1}^{\tau} \1{Z_n = y} = \frac{1}{\bE Y_1} \bE \sum_{n=1}^{Y_1} \1 { Z_n = y } 
= \frac{1}{\bE Y} \bE  \1 { Y \geq y } ,
\]
which is exactly $\pi$ given by~\eqref{eq:pi-k}. 
As remarked after~\eqref{eq:overshoot-chain}, the Markov chain is irreducible and
aperiodic under the hypotheses of the proposition, and 
so the convergence in~\eqref{eq:pi-limit}
follows from the Markov chain convergence theorem.
Part~\ref{prop:walk-overshoot-ii} 
can be found as Theorem 6.10.3 of~\cite[pp.~103--4]{gut-srw}, and
may also be derived from the  
classical 
renewal theorem for aperiodic lattice random variables~\cite[p.~363]{feller2}.  

Part~\ref{prop:walk-overshoot-iii} follows from  part~\ref{prop:walk-overshoot-i}, and indeed from Theorem~3 of Lorden~\cite{lorden} (see also Theorem 3.1 of~\cite{janson}),
once one knows that $\bE ( Y^{p-1} ) < \infty$ whenever $\bE ( |X|^p ) < \infty$.
This fact (the $p>2$ analogue of Spitzer's result for $p=2$ that we already used) 
 is provided by results of Doney~\cite{doney} and Lai~\cite{lai}.
\end{proof}

\subsection{Estimates from Kesten on hitting before exit}
\label{sec:kesten}

In this section we present some estimates on hitting and exit of
$\Z$-valued random walks, derived more-or-less directly
from fine results of Kesten~\cite{kesten1961a,kesten1961b}.
To state the results,
for $t \in \Z$ and $A \subset \Z$, define
\begin{equation}
\label{eq:T-eta-S}
  T_t := \inf \{ n \in \ZP: S_n = t \}, \text{ and } 
   \eta_A := \inf \{ n \in \ZP: S_n \notin A \},
\end{equation}
respectively the first hitting time of~$t$ and the first exit time from~$A$
for the random walk~$S$. 
First, we describe informally the results that we will use.
Throughout this discussion, we assume that~\eqref{ass:irreducible} holds and that $\Exp X =0$.

\begin{itemize}
\item
An easy, but important, 
consequence of  recurrence and irreducibility
is a \emph{local hitting} estimate saying that
there is high probability of visiting a nearby site before going far away:
see Lemma~\ref{lem:recurrence-hitting} below.
\item Kesten provides general \emph{gambler's ruin} estimates 
on the probability of exiting a large interval on one side rather than the other: see Lemma~\ref{lem:kesten-gambler} below. 
\item The local hitting and gambler's ruin estimates show no essential distinction between the
finite- and infinite-variance cases. Where the distinction arises is in what the walk does when it exits an interval, i.e., the behaviour of \emph{overshoots}, as we have seen with the contrast between
Propositions~\ref{prop:walk-overshoot} (tight overshoots) and~\ref{prop:dynkin-lamperti} (large-scale overshoots). 
\item Kesten combines ingeniously the above elements to obtain precise asymptotics for the probability of hitting a particular point before exit from a large interval. Roughly speaking, in the finite-variance case, the tight overshoots and local hitting estimates show that the gambler's ruin probabilities capture the essential behaviour, while in the infinite-variance case there may be many overshoots of the target point before it is successfully hit. Lemmas~\ref{lem:kesten-visit-before-exit} and~\ref{lem:Bounds_on_q} below present the main estimates we will need of this type.
\end{itemize}

We now give precise statements of the results described loosely above. 
We start with the following local hitting property, which is a consequence of recurrence and irreducibility: see equation{~(2.5) of~\cite[p.~248]{kesten1961a}}.

\begin{lemma}
    \label{lem:recurrence-hitting}
Suppose that~\eqref{ass:irreducible} holds and $\Exp X =0$. Then
$\lim_{N \to \infty} \bP_0 ( T_k < \eta_{[-N,N]} ) = 1$
    for every fixed $k\in \Z$.
\end{lemma}

Next we present Kesten's general gambler's ruin estimates.
The finite-variance part, Lemma~\ref{lem:kesten-gambler}\ref{lem:kesten-gambler-i},
 is contained in Theorem~2 of~\cite[p.~256]{kesten1961a}, while the infinite-variance part, Lemma~\ref{lem:kesten-gambler}\ref{lem:kesten-gambler-ii}, is contained in Corollary~1 of~\cite[p.~273]{kesten1961b}. We remark that the strength of part~\ref{lem:kesten-gambler-i} is that no more than finite second moments is assumed, but asymptotic lower and upper bounds match; 
 compare e.g.~Theorem 5.1.7 of~\cite[p.~127]{llbook}, which does not provide matching bounds, or Theorem~5 of~\cite{lotov}, which requires the hypothesis $\Exp ( |X|^3 ) < \infty$.

 \begin{lemma}[Kesten 1961~\cite{kesten1961a,kesten1961b}]
 \label{lem:kesten-gambler}
 Suppose that~\eqref{ass:irreducible} holds.
 \begin{thmenumi}[label=(\roman*)]
\item\label{lem:kesten-gambler-i} Suppose that $\Exp ( X^2 ) < \infty$ and $\Exp X =0$. Then,
for every $c \in \RP$,
 \begin{equation}
     \label{eq:right-exit-kesten-2}
     \lim_{N \to \infty} \bP_0 \bigl( S_{\eta_{[-cN, N]}} > N \bigr) 
     = \frac{c}{1+c} .
 \end{equation}
\item\label{lem:kesten-gambler-ii} Suppose that~\eqref{ass:stable-alpha} holds with $\alpha \in (1,2)$.
Then, for every $c \in \RP$,
 \begin{equation}
     \label{eq:right-exit-kesten}
     \lim_{N \to \infty} \bP_0 \bigl( S_{\eta_{[-cN, N]}} > N \bigr) 
     = \frac{\Gamma(\alpha)}{\Gamma (\alpha/2)^2} \int_{(1+c)^{-1}}^1 u^{\frac{\alpha}{2}-1} (1-u)^{\frac{\alpha}{2}-1} \ud u .
 \end{equation}
  \end{thmenumi}
\end{lemma}
\begin{remarks}
\label{rems:kesten-gambler}
\begin{myenumi}[label=(\roman*)]
\item\label{rems:kesten-gambler-i}
A change of variable followed by 
the symmetric beta-integral formula~\cite[p.~258]{as} 
shows that, for every $\alpha >0$,
\begin{equation}
    \label{eq:beta}
   2 \int_{1/2}^1{u^{\frac{\alpha}{2}-1} (1-u)^{\frac{\alpha}{2}-1} \ud u}
 =    \int_0^1{u^{\frac{\alpha}{2}-1} (1-u)^{\frac{\alpha}{2}-1} \ud u}
    = \frac{\Gamma (\alpha/2)^2}{\Gamma(\alpha)} .
\end{equation}
Consequently, for exit from a symmetric interval ($c=1$),
both~\eqref{eq:right-exit-kesten-2} and~\eqref{eq:right-exit-kesten}
yield the asymptotically-fair ruin estimate $\lim_{N \to \infty}\bP_0 \bigl( S_{\eta_{[-N, N]}} > N \bigr) = 1/2$.
\item\label{rems:kesten-gambler-ii}
Corollary~1 in Kesten~\cite[p.~273]{kesten1961b} is stated slightly differently from \eqref{eq:right-exit-kesten}, in terms of $\lim_{N \to \infty}\bP_0 \bigl( S_{\eta_{[-cN, N]}} <-cN \bigr)$. 
However, one  easily deduces~\eqref{eq:right-exit-kesten} using~\eqref{eq:beta}, since
\begin{align*}
&\lim_{N \to \infty} 
\bP_0 \bigl( S_{\eta_{[-cN, N]}} > N \bigr)
   =  1 - \lim_{N \to \infty} \bP_0 \bigl( S_{\eta_{[-cN, N]}} <-cN \bigr)\\
 &=  \frac{\Gamma(\alpha)}{\Gamma (\alpha/2)^2} \left[ \int_0^1 {u^{\frac{\alpha}{2}-1} (1-u)^{\frac{\alpha}{2}-1} \ud u}- \int_0^{(1+c)^{-1}}{u^{\frac{\alpha}{2}-1} (1-u)^{\frac{\alpha}{2}-1} \ud u} \right]\\
 &=  \frac{\Gamma(\alpha)}{\Gamma (\alpha/2)^2} \int_{(1+c)^{-1}}^1 u^{\frac{\alpha}{2}-1} (1-u)^{\frac{\alpha}{2}-1} \ud u.
\end{align*}
\end{myenumi}
\end{remarks}

We turn to the most delicate results, which are estimates for the probability that a particular site in an interval is visited before {exiting} the interval. 
Borrowing notations from {Kesten~\cite[p.~275]{kesten1961b}}, let us  define for $N \in \ZP$, $c >0$, 
and $k \in \{0,1,\ldots, N \}$,
\begin{equation}
\label{Eq: def_q_N}
q_{\alpha, N} \Bigl( \frac{k}{N}; c \Bigr) :=  \bP_k \left(  T_{0} < \eta_{[ -cN, N ]} \right), 
\end{equation} 
where $\alpha =2$ if $\Exp ( X^2 ) < \infty$ and $\alpha \in (1,2)$ if~\eqref{ass:stable-alpha} is satisfied. Also  define $q_{\alpha,N}( y ; c)$ over all $y \in [0,1]$ by linear interpolation, i.e.,
 \[ q_{\alpha, N} ( y ; c ) := (  k+1-N y ) q_{\alpha, N}\Bigl( \frac{k}{N}; c \Bigr) +  (Ny-k) q_{\alpha, N}\Bigl( \frac{k+1}{N}; c \Bigr) , \text{ if } \frac{k}{N} < y < \frac{k+1}{N}. \]
For $c' \geq c >0$, we have $\eta_{[-cN,N]} \leq \eta_{[-c'N,N]}$, a.s.; this shows  the following monotonicity property 
\begin{equation}
\label{eq:q-monotone}
q_{\alpha, N} ( y ; c ) \leq q_{\alpha, N} ( y ; c' ), \text{ whenever } 0 < c \leq c'.
\end{equation}

The main result of this subsection is the following.
The result is essentially due to Kesten~\cite{kesten1961a,kesten1961b}, but part~\ref{lem:kesten-visit-before-exit-i} is not given explicitly by Kesten, so we give a proof later in this subsection.

\begin{lemma}[Kesten 1961~\cite{kesten1961a,kesten1961b}]
 \label{lem:kesten-visit-before-exit}
Suppose that~\eqref{ass:irreducible} holds. Fix  $0 \le y <1$ and $c>0$.  
 \begin{thmenumi}[label=(\roman*)]
\item\label{lem:kesten-visit-before-exit-i} Suppose that $\Exp ( X^2 ) < \infty$ and $\Exp X =0$. 
 Then there exists the limit
 \begin{equation}
 \label{eq:q_limit-2}
 \lim_{N \to \infty}q_{2, N}(y; c) = q_{2} (y; c)  :=     1-y.
 \end{equation}
\item\label{lem:kesten-visit-before-exit-ii} Suppose that~\eqref{ass:stable-alpha} holds with $\alpha \in (1,2)$. Then there exists the limit
 \begin{align}
 \label{eq:q_limit-alpha}
 \lim_{N \to \infty}q_{\alpha, N}(y; c) & = q_{\alpha} (y; c), \end{align}
 where
  \begin{align}
 \label{eq:q_limit_integral}
 {q_{\alpha} (y ; c)} & :=
 \left(\alpha -1\right) c^{1-\frac{\alpha}{2}}\left(1+c\right)^{\alpha -1 }\left(y+c  \right)^{\frac{\alpha}{2}}y^{\alpha -1}
  \int_y^1  \left(y+cv\right)^ {-\alpha}  \left(1-v\right)^{\frac{\alpha}{2} -1}\ud v.
 \end{align}
  \end{thmenumi}
\end{lemma}

\begin{remark}
\label{rem:q-alpha-def}
Formula~\eqref{eq:q_limit_integral} defines $q_{\alpha} (y ; c)$ when $y >0$,
and when $y=0$ the definition is to be understood as the limit $q_\alpha (0; c) := \lim_{y \to 0} q_\alpha (y ;c) =1$ (as can be verified by calculus).
Moreover, there is continuity as $\alpha \uparrow 2$ in the sense that
$\lim_{\alpha \to 2} q_\alpha (y,c) = 1-y$ to match with~\eqref{eq:q_limit-2}. See Lemma~\ref{lem:Bounds_on_q} below for proofs of these properties.
\end{remark}

We defer the proof of Lemma~\ref{lem:kesten-visit-before-exit} until the end of this section. First, we need some technical results on the equicontinuity of the $q_{\alpha,N} (y;c)$ appearing in Lemma~\ref{lem:kesten-visit-before-exit},
as well as the corresponding quantities that appear in Lemma~\ref{lem:kesten-gambler}.
In the latter case, we need a little more notation.
Similarly to $q_{\alpha, N} ( y ; c )$, we can define for every $y \in [0,1]$, the exit 
 probabilities 
\begin{equation*}
\label{Eq:def_p_tilde}
p_{\alpha, N} \Bigl( \frac{k}{N} \Bigr) := \bP_k \bigl( S_{\eta_{[ 0,N]}} < 0  \bigr), \text { for } k \in \{ 1,2,\ldots, N \}.
\end{equation*}
For general $\frac{k}{N} < y < \frac{k+1}{N}$, we define $p_{\alpha,N} (y)$ via linear interpolation, similarly to $q_{\alpha, N} ( y ; c )$.
The equicontinuity results that we need are as follows. 

\begin{lemma}[Kesten 1961~\cite{kesten1961b}]
\label{lem:kesten-equicontinuity}
 Suppose that~\eqref{ass:irreducible} holds. For $1 < \alpha \leq 2$, suppose in addition that (if $\alpha = 2$) $\Exp ( X^2 ) < \infty$ and $\Exp X =0$, or (if $1<\alpha<2$) that~\eqref{ass:stable-alpha} holds.  Fix  $0 \le \lambda <1$ and $c>0$. The following hold.
\begin{thmenumi}[label=(\roman*)]
\item
\label{lem:kesten-equicontinuity-i}
The family of functions $\left(p_{\alpha, N} ( y )\right)_{N \in \N}$  of $y \in [0,\lambda]$ is    uniformly equicontinuous.
\item\label{lem:kesten-equicontinuity-ii} The family of functions $\left(q_{\alpha, N} ( y ; c )\right)_{N \in \N}$  of $y \in [0,\lambda]$ is    uniformly equicontinuous.
\end{thmenumi}
\end{lemma}

Part~\ref{lem:kesten-equicontinuity-ii} is available explicitly in Kesten~\cite{kesten1961a,kesten1961b}; we give a proof of
part~\ref{lem:kesten-equicontinuity-i}, using similar ideas, in Appendix~\ref{sec:appendix-equicontinuity}. 
To exemplify the usefulness of 
Lemma~\ref{lem:kesten-equicontinuity}, we state two of its consequences which extend the convergence stated in Lemma~\ref{lem:kesten-visit-before-exit}.

First, suppose that $y_N \in [0,1)$ is a sequence such that $\lim_{N \to \infty} y_N = y \in [0,1)$. Take $\lambda \in (y,1)$. Then uniform equicontinuity means that for every $\eps>0$ there exists $\delta >0$ such that $| q_{\alpha,N} (y;c) - q_{\alpha,N} (y';c) | \leq \eps$ whenever $y , y'\in [0,\lambda]$ and $|y'-y| \leq \delta$. In particular, for all $N$ large enough, we have $|y-y_N| \leq \delta$.
Consequently, equicontinuity extends the convergence in~\eqref{eq:q_limit-2} and~\eqref{eq:q_limit-alpha} to
\begin{equation}
    \label{eq:q-limit-y-limit}
    \lim_{N \to \infty} q_{\alpha,N} (y_N ; c) = q_\alpha ( y ; c),
\end{equation}
whenever $y_N \to y \in [0,1)$ and the relevant hypotheses from Lemma~\ref{lem:kesten-visit-before-exit} hold.

Here is a second consequence. For fixed $\alpha, c$, the family $( q_{\alpha, N} ( y ; c) )_{N \in \N}$
is uniformly equicontinuous, as  functions of $y \in A$ for  any compact $A \subset [0,1)$.
Hence, by~\eqref{eq:q_limit-alpha}, $q_{\alpha,N} (y; c)$
converges uniformly as $N \to \infty$ to $q_{\alpha} (y;c)$, as functions of $ y \in A$. 
In particular, for every  compact $A \subset [0,1)$, 
\begin{equation}
    \label{eq:q-inf-convergence}
 \lim_{N \to \infty} \inf_{y \in A } q_{\alpha,N} (y ; c ) = \inf_{y \in A} q_{\alpha} (y ; c ) . \end{equation}
We state~\eqref{eq:q-limit-y-limit} and~\eqref{eq:q-inf-convergence} for $q_{\alpha,N}$; analogous statements for $p_{\alpha,N}$ are deduced in the same way.

We will need the following bounds on  $q_{\alpha} (y; c)$;
in the proof we make a first use of the equicontinuity
from Lemma~\ref{lem:kesten-equicontinuity}.
 
 \begin{lemma}
 \label{lem:Bounds_on_q}
  Suppose that~\eqref{ass:irreducible} holds, and
  that~\eqref{ass:stable-alpha} holds with $\alpha \in (1,2)$. Let $q_{\alpha} (y; c)$ be as defined in~\eqref{eq:q_limit_integral}. 
Then it holds that, for every  $c>0$ and all $0 \leq y < 1$,
  \begin{equation}
  \label{eq:lower_bound_q}
      q_\alpha (y ; c) \geq (\alpha -1 ) c^{1-\frac{\alpha}{2}} ( c+y)^{\frac{\alpha}{2} -1} ( 1-y),\end{equation}
      and, moreover, for every $c >0$,
 \begin{equation}
  \label{eq:q-y-at-0}
      \lim_{y \to 0} q_\alpha (y ; c)  = 1. \end{equation}
       On the other hand, for every $\delta>0$ it holds that
 \begin{equation}
 \label{eq:upper_bound_q}
 \sup_{\delta \leq y \leq 1} q_{\alpha} (y; c) < 1.
 \end{equation}
 \end{lemma}
 \begin{proof}
 Similarly to~\eqref{eq:q_limit_integral}, define
\begin{equation}
    \label{eq:u-def}
    u_\alpha (y ; c) := \left(\alpha -1\right) c^{1-\frac{\alpha}{2}}\left(1+c\right)^{\alpha -1 }\left(y+c  \right)^{\frac{\alpha}{2}}y^{\alpha -1} \int_y^1  \left(y+cv\right)^ {-\alpha}  \ud v.
\end{equation}
Some calculus shows that
\begin{align*}
 \int_y^1 \left(y + cv \right)^{-\alpha} \ud v  = \frac{y^{1-\alpha}}{c (\alpha-1)} \left[ \left(1+c\right)^{1-\alpha} - \left(1+\frac{c}{y} \right)^{1-\alpha} \right].
\end{align*}
Hence, by~\eqref{eq:u-def}, we obtain
\begin{align}
\label{eq:u-expression}
     u_\alpha (y ; c) =   \left(\frac{c+y}{c}\right)^{\frac{\alpha}{2}} 
     \left[ 1 - \left(\frac{1+c}{1+(c/y)} \right)^{\alpha - 1} \right].
\end{align}
Since $0< \alpha /2 < 1$, for all $v \in [0,1)$ 
it holds that $(1-v)^{(\alpha/2)-1} \geq 1$. 
Comparison of~\eqref{eq:q_limit_integral} and~\eqref{eq:u-def}
then shows that $q_\alpha (y ; c) \geq u_\alpha (y; c)$
for all $y \in [0,1]$.
Using the fact that $(1 - x)^{\alpha -1} \leq 1 - (\alpha-1) x$ for $\alpha \in [1,2]$ and {$x \in [0,1]$}, we have that
\[ \left(\frac{1+c}{1+(c/y)} \right)^{\alpha - 1}  = 
\left( 1 - \frac{c(1-y)}{c+y} \right)^{\alpha - 1}
\leq 1 - (\alpha -1 ) \left( \frac{c(1-y)}{c+y} \right).
\]
Combining the last bound with~\eqref{eq:u-expression} yields the bound in~\eqref{eq:lower_bound_q}. Moreover, since for $0 < y \leq 1$ the first factor in the expression on the right-hand side of~\eqref{eq:u-expression} is at least~$1$, and the second factor (in square brackets) is non-negative, it also follows that 
\[ \liminf_{y \to 0} q_\alpha (y ; c) \geq \liminf_{y \to 0} u_\alpha (y ; c) \geq 1 - \limsup_{y \to 0} \left(\frac{1+c}{1+(c/y)} \right)^{\alpha - 1} =1,\]
for every $c >0$ and  $\alpha \in (1,2)$. This yields~\eqref{eq:q-y-at-0}, since $0 \le q_\alpha (y ; c) \le 1$.

Finally, we obtain the bound~\eqref{eq:upper_bound_q} by an application
of Kesten's gambler's ruin estimate~\eqref{eq:right-exit-kesten}.
Observe that, for $c \geq 0$ and $k \in \N$,
\begin{align*}
\bP_k \bigl(  T_{0} < \eta_{[ -cN, N ]} \bigr) 
&
\leq \bP_k \bigl(  S_{\eta_{[ 1, N]}} < 1 \bigr) 
= 1 - \bP_k \bigl(  S_{\eta_{[ 1, N ]}} > N \bigr),
\end{align*}
since in order to visit~$0$ before exiting the interval $[-cN, N]$,
the walk must exit the (smaller) interval $[1, N]$ on the left. 
In particular, for $y \in [0,1)$ and  a sequence $k_N \in \N$ such that $\lim_{N \to \infty} k_N/N = y$, then
using equicontinuity via~\eqref{eq:q-limit-y-limit},
\begin{align*}
 q_\alpha (y; c) & = \lim_{N \to \infty} \bP_{k_N} \left(  T_{0} < \eta_{[ -cN, N ]} \right)  .
 \end{align*}
Suppose $y \in [0,1)$ is rational; then we can choose the sequence $k_N$ such that $k_N = y N$ for a subsequence of $N$. Hence
\begin{align*}
  q_\alpha (y; c) & \leq 1 - \liminf_{N \to \infty}  \bP_{y N} \left(  S_{\eta_{[ 1, N ]}} > N \right) \\
& = 1 - \liminf_{N \to \infty}  \bP_{0} \left(  S_{\eta_{[1- y N, (1-y) N ]}} > (1-y)N \right) \\
& \leq 1 - \liminf_{N \to \infty}  \bP_{0} \left(  S_{\eta_{[ -c' N,  N ]}} > N \right), 
\end{align*}
for every $c' < c_y$ where $c_y:= \frac{y}{1-y} \in (0,\infty)$.
Then $(1+c_y)^{-1} = 1- y$
for $y \in [0,1]$, and hence, 
from~\eqref{eq:right-exit-kesten},
for every rational $y \in [\delta, 1]$,
\begin{align*}
 q_\alpha (y; c)  
& \leq 1 - 
\inf_{\delta \leq y \leq 1} \frac{\Gamma(\alpha)}{\Gamma (\alpha/2)^2} \int_{1-y}^1 u^{\frac{\alpha}{2}-1} (1-u)^{\frac{\alpha}{2}-1} \ud u \\
& =1 -  \frac{\Gamma(\alpha)}{\Gamma (\alpha/2)^2} \int_{1-\delta}^1 u^{\frac{\alpha}{2}-1} (1-u)^{\frac{\alpha}{2}-1} \ud u ,
\end{align*}
which yields~\eqref{eq:upper_bound_q}, using continuity of $y \mapsto q_\alpha (y;c)$ over $[\delta,1]$.
\end{proof}

We conclude this section with a proof of Lemma~\ref{lem:kesten-visit-before-exit}. Here and subsequently, we write $\cF_n := \sigma (S_0, S_1, \ldots, S_n)$
for the $\sigma$-algebra ($\cF_n \subseteq \cF'$) generated by the first $n$ steps of the random walk.

\begin{proof}[Proof of Lemma~\ref{lem:kesten-visit-before-exit}]
Part~\ref{lem:kesten-visit-before-exit-ii}  is Theorem~2 in~\cite[p.~277]{kesten1961b}.
Part~\ref{lem:kesten-visit-before-exit-i} is not explicitly stated in~\cite{kesten1961a,kesten1961b},
but can be deduced from results therein, as we now demonstrate.

 Suppose that $0 < y < 1$, and
set $\lambda := \inf \{ n \in \ZP : S_n < 0\}$, which is a stopping time with respect to filtration~$\cF_n$.
Under either of the hypotheses of the lemma, we have $\Exp X =0$ and hence
$\bP_k ( \lambda < \infty) = 1$ for every $k \in \Z$, by recurrence. 
Suppose  that $k_N \in \N$ 
is a sequence such that $k_N / N \to y$ as $N \to \infty$, and consider events
\begin{align*}
E_1 & := E_1 (N) := \{ S_{\eta_{[ 0,N]}} < 0 \}, \\
E_2 & := E_2 (A) := \{ S_\lambda \in [  - A, 0] \} ,\\
E_3 & := E_3 (N, c) := \{ T_{0} < \eta_{[-cN  , N ]} \},
\end{align*}
and note that both $E_1$ and $E_2$ are $\cF_\lambda$-measurable.
By equicontinuity, similarly to~\eqref{eq:q-limit-y-limit} but applied to $p$, we have that
\begin{align*}
  \lim_{N \to \infty}  \bP_{k_N} (E_1) & = 1 - \lim_{N \to \infty} \bP_{k_N} \bigl( S_{\eta_{[0, N]}} >  N \bigr) 
  = \lim_{N \to \infty} p_{2,N} ( y).
\end{align*}
For rational $y$, there is a sequence of $N_m$, $m \in \N$, for which $y N_m \in \N$, and then
\begin{align*}
\lim_{N \to \infty} p_{2,N} ( y) = \lim_{m \to \infty} \bP_{y N_m} ( S_{\eta_{[0,N_m]}} < 0 )
& = 1 - \lim_{m \to \infty} \bP_0 ( S_{\eta_{[-y N_m, (1-y)N_m ]}} > (1-y) N_m) \\
& = 1 - \lim_{N \to \infty} \bP_0 ( S_{\eta_{[-c_y N, N ]}} > N) = 1- y,
\end{align*}
where $c_y := \frac{y}{1-y}$ and we have used 
Lemma~\ref{lem:kesten-gambler}\ref{lem:kesten-gambler-i} for the convergence. 
In other words, $\lim_{N \to \infty}  p_{2,N} (y) = 1-y$
for all rational~$y$. Moreover,  
we have from
Proposition~\ref{prop:walk-overshoot} that, since $y>0$, for every $\eps >0$, there exists $A < \infty$ such that
$\lim_{N \to \infty} \bP_{k_N} (E_2) \geq 1- \eps$.
Started from a site in $[  - A, 0]$, the probability that the walk visits
$0$ before exit from
$[-cN , N]$
tends to $1$, by Lemma~\ref{lem:recurrence-hitting}, i.e., on event $E_2$,
\[ \bP
( E_3  \mid \cF_\lambda )   
\geq \inf_{z \in [-A,0]} \bP_z ( T_{0} < \eta_{[-cN,N]} )
\to 1,
\]
as $N \to \infty$. Hence, given $\eps >0$ we can choose $A$ large enough and then $N$ large enough 
so that $\bP
( E_3  \mid \cF_\lambda ) \geq 1- \eps$ on $E_2$, that
$\bP_{k_N} (E_1) \geq 1-y -\eps$, and that $\bP_{k_N} (E_2) \geq 1- \eps$.
Then, by the strong Markov property at time $\lambda$,
\begin{align*}
\bP_{k_N} ( T_{0} < \eta_{[-cN  , N ]} ) 
& \geq \bE_{k_N} \left[ \bP ( E_3 \mid \cF_\lambda )  \2{ E_1 \cap E_2}  \right] \\
& \geq (1-\eps) \bP_{k_N} ( E_1) - \bP_{k_N} (E_2^\rc ) \\
& \geq (1-\eps) (1-y -\eps) - \eps \geq 1 - y - 3 \eps.
\end{align*}
Since $\eps >0$ was arbitrary, we deduce that
$\liminf_{N \to \infty}\bP_{k_N} ( T_{0} < \eta_{[-cN  , N ]} )  \geq 1-y$.
Combined with a similar argument in the other direction,
we verify~\eqref{eq:q_limit-2}, with the restriction that $y>0$,
but this restriction is easily removed by a continuity argument, as in Kesten~\cite[p.~275]{kesten1961b}.
  \end{proof}

\section{Proofs of the  bounds on cluster growth}
\label{sec:inner-radius}

\subsection{Lower bounds: Overview and some heuristics}
\label{sec:inner-radius-heuristics}

Recall that $r_m$
defined in equation~\eqref{eq:max-radius}
is the maximal $r \in \ZP$ such that $[-r,r] \cap \Z$ is contained in $\Cl_m$. The purpose of this section is to study the asymptotics of $r_m$. In \S\ref{sec:inner-radius-finite-variance}, we prove~\eqref{eq:inner_radius-lim} of Theorem \ref{thm:idla-light-tail}, which covers the case where increments have finite variance. To prove \eqref{eq:r-m-heavy} of Theorem~\ref{thm:idla-heavy-tail}, the infinite variance case,  we prove the lower bound in \S\ref{sec:inner-radius-infinite-variance-upper-bound} and the upper bound in \S\ref{sec:inner-radius-infinite-variance-lower-bound}.

In view of the bound $\limsup_{m \to \infty} r_m/m \leq 1/2$, a.s.,
immediate from~\eqref{eq:r-R-trivial},
to prove~\eqref{eq:inner_radius-lim}
it suffices to prove that $\liminf_{m \to \infty} r_m/m \geq 1/2$, a.s. It turns out to be more convenient at this point to work with 
the coverage times
\begin{equation}
\label{eq:sigma-x-def}
\sigma_x := \inf \{ m \in \ZP : r_m \geq x \} .
\end{equation}
The sequences $r_m$ and $\sigma_x$
    are related by the inversion:
    \begin{equation} \label{eq:inversion}
     \text{It holds that $r_{m} \geq x$ if and only if $\sigma_{x} \leq m$}. 
         \end{equation}
In particular,  for $c \in (0,\infty)$, equivalent to the statement
    $\liminf_{m\to\infty} (r_m/m) \geq c$, a.s., is the statement $\limsup_{x \to \infty} (\sigma_x/x ) \leq 1/c$, a.s. 
Hence to prove~\eqref{eq:inner_radius-lim} in Theorem~\ref{thm:idla-light-tail},
it is enough to prove the following.
 
 \begin{proposition}
 \label{prop:inner-radius-inverse}
Suppose that~\eqref{ass:irreducible} holds,  $\Exp ( X^2) < \infty$, and $\Exp X =0$. 
 Then, a.s.,
\begin{equation}
\label{eq:sigma_limsup}
   \limsup_{x \to \infty} \frac{\sigma_x}{x} \leq 2.
\end{equation}
 \end{proposition}

In the finite-variance case, 
the trivial bound $\sigma_x/x \geq 2$ that follows from~\eqref{eq:r-R-trivial} means that 
to prove Theorem~\ref{thm:idla-light-tail},
the upper bound~\eqref{eq:sigma_limsup} is sufficient. 
 In the infinite-variance case, for Theorem~\ref{thm:idla-heavy-tail}
 we need not only an upper bound, presented in Proposition~\ref{prop:upper-bound_infinite-variance} that follows, but also   non-trivial lower bounds,
which are the subject of \S\ref{sec:inner-radius-infinite-variance-lower-bound}. Define the constant
\begin{equation}
    \label{eq:C-alpha-dash}
   C_\alpha' :=  (\alpha -1)^{-1} (4-\alpha)(3-\alpha)^3 (2-\alpha)^{\alpha-2} . 
\end{equation}

 \begin{proposition}
 \label{prop:upper-bound_infinite-variance}
    Suppose that~\eqref{ass:irreducible} holds, 
and that~\eqref{ass:stable-alpha} holds with $\alpha \in (1,2)$. Then, for $C_\alpha' \in (2,\infty)$ 
given by~\eqref{eq:C-alpha-dash}, it holds that, 
a.s.,
\begin{equation}
\label{eq:sigma_limsup_infinite_variance_upper_bound}
\limsup_{x \to \infty} \frac{\sigma_x}{x} \leq C'_\alpha.
\end{equation}
\end{proposition}

The outline of the proofs of the upper bounds in both Propositions~\ref{prop:inner-radius-inverse} and~\ref{prop:upper-bound_infinite-variance}, has similarities to the corresponding argument in~\cite[\S 3]{lbg}, with
the estimates of \S\ref{sec:kesten} (derived from Kesten~\cite{kesten1961a,kesten1961b}) providing the central probabilistic components.

 Recall from \S\ref{sec:model} that $S^{(j)}$ denotes the $j$th random walk in the IDLA process,
 and that $\tau_j$, defined in equation~\eqref{eq:cluster expansion}, is the internal time index of the walk $S^{(j)}$ when it first exits the prior cluster~$\Cl_{j-1}$. In this section we develop several arguments using the sequence of random walks $S^{(j)}$. We note that the IDLA process depends only on each $S^{(j)}$ up to its associated time~$\tau_j$,  but in fact the walk $S^{(j)}$ is defined for all time, and several arguments make use of this, for example, to overcome certain dependence. In arguments that emphasize the   finite-walk perspective, we sometimes describe walk $S^{(j)}$ as being \emph{active} up until time $\tau_j$, and then \emph{terminating}, while in arguments that use the full trajectory of the walk, we sometimes refer to the walk $S^{(j)}$ as \emph{indefinitely extended}.
 
Let us describe, informally, how the results of \S\ref{sec:kesten} 
enable us to understand the behaviour of 
$\sigma_x$.
Consider time $\sigma_x$, so that the interval $[-x,x]$
is fully occupied by the cluster $\Cl_{\sigma_x}$.
We fix $u >1$ 
and study the IDLA process up to the time at which the interval $[-ux, ux]$
is fully occupied, in order to bound from above $\sigma_{ux} - \sigma_x$. 
We must argue that unoccupied sites are filled rather rapidly by subsequent walkers. 
The outline of the argument is as follows. 
\begin{itemize}
\item[1.] The gambler's ruin estimates of Lemma~\ref{lem:kesten-gambler}
    show that any subsequent walker will exit $[-x,x]$
    on the right or left each with probability almost $1/2$, and, since $[-x,x]$
    is fully occupied, the walker will still be active when it does so.
    This is the case in both the settings of Propositions~\ref{prop:inner-radius-inverse} and~\ref{prop:upper-bound_infinite-variance}.
    \item[2a.]
    Consider some site $t \in [x, ux] \cap \Z$.
    In the finite-variance case (Proposition~\ref{prop:inner-radius-inverse})
    by tightness of overshoots (Proposition~\ref{prop:walk-overshoot}\ref{prop:walk-overshoot-ii})
    and recurrence (Lemma~\ref{lem:recurrence-hitting}), a random walk that exits $[-x,x]$ on the right, will, with probability close to~$1$ (if $u \approx 1$) visit $t$ before exiting from the interval $[-x, s x]$, where $s >u$. Hence (using point~1 above) the probability that the walk exits $[-x,x]$ on the right \emph{and} visits $t$ before exit from $[-x, s x]$ is approximately $q \approx 1/2$.
    \item[2b.]
Consider some site $t \in [x, ux] \cap \Z$. In the infinite-variance case  (Proposition~\ref{prop:upper-bound_infinite-variance})  a random walk that exits $[-x,x]$ on the right will visit $t$ before exiting from the interval $[-x, s x]$, where $s>u$, with a   probability bounded below by some strictly positive $q = q_\alpha (u,s) \in (0,1/2)$, by  Lemma~\ref{lem:kesten-visit-before-exit} and the bounds in Lemma~\ref{lem:Bounds_on_q}.
    \item[3.]
    To obtain statements that hold with very high probability, from the walk-by-walk probability statements in 2a and 2b, we consider the $k = A x$
 walkers directly after time $\sigma_x$ and exploit binomial concentration. The number
    of these walkers that exit $[-x,x]$ on the right and then visit  $t \in [x, ux] \cap \Z$
    before exiting the interval $[-x, s x]$ is binomial, and, by binomial concentration, we can, by balancing small constants, ensure is, with very high probability, at least about $A q x$.
    \item[4.]
    There are at most $(s-1)x$ sites of $\Z$ in the interval $[x,sx]$, and each can accommodate at most one walker adding to the aggregate, so if  $Aq  > s -1$, then at least one walker is active when it visits $t$. Hence $t$ is contained in the cluster at time about $\sigma_x + A  x$, where the optimal $A$ is $A \approx (s-1)/q$. (See Lemma~\ref{lem:over-filling} for a precise version of this argument.)
    \item[5.] The binomial concentration is sufficient to show that this occurs with high probability for all $t \in [-ux, ux] \cap \Z$, i.e., with high probability
    by time $\approx \sigma_x +  (s-1) x/q $ we cover $[-ux , ux]$.
This shows $\sigma_{ux} - \sigma_x \leq  (s-1)x/q$, with high probability. 
With an interpolation argument, this translates to a bound of the form
     \begin{equation}
     \label{eq:outline-bound}
     \limsup_{x \to \infty} \frac{\sigma_{x }}{x} \leq \frac{1}{q} \left(\frac{s-1}{u-1} \right) u, \as; \end{equation}
     see~\eqref{eq:s-u-algebra} (for the finite variance case)
     and~\eqref{eq:s-u-algebra-alpha} (infinite variance) below.
Choosing constants carefully, in the case of Proposition~\ref{prop:inner-radius-inverse}
we can take $s \approx u \approx 1$ and $q \approx 1/2$, giving the constant~$2$ in the bound~\eqref{eq:sigma_limsup}.
In the setting of Proposition~\ref{prop:upper-bound_infinite-variance}, 
choosing $s >u >1$ appropriately leads to the upper bound in~\eqref{eq:sigma_limsup_infinite_variance_upper_bound} with the constant $C_\alpha'$ from~\eqref{eq:C-alpha-dash}.
\end{itemize}

\subsection{Preliminaries}
\label{sec:inner-radius-preliminaries}

Analogously to the notation in~\eqref{eq:T-eta-S},
for the random walk $S^{(j)}$, $j \in \N$,
define the first hitting time $T^{(j)}_t$
of   $t \in \Z$ by
\begin{equation}
    \label{eq:T-j-def}
    T^{(j)}_t := \inf \{ n \in \ZP: S^{(j)}_n = t \},
\end{equation}
and  the first exit time $\eta^{(j)}_A$ from the set $A \subseteq \Z$ by
\begin{equation}
    \label{eq:eta-j-def}
    \eta^{(j)}_A := \inf \{ n \in \ZP: S^{(j)}_n \notin A \}.
\end{equation}
Note that, $\tau_m$ as defined in equation~\eqref{eq:cluster expansion}
has the representation $\tau_m = \eta^{(m)}_{\Cl_{m-1}}$ in the notation in equation~\eqref{eq:eta-j-def}.
For $t \in \Z$ and $x \in \N$, define
\begin{equation}
    \label{eq:N-plus-def}
    N^+_{m,k} (t ; x) := \sum_{j=m+1}^{m+k} \1 { S^{(j)}_{\eta_{[-x,x]}^{(j)}} > x, \, T^{(j)}_t \le \eta^{(j)}_{\Cl_{j-1}}} ,
\end{equation}
the number of the next $k$ random walks, released after time~$m$,
that (i) exit the interval $[-x,x]$ on the right, and (ii) visit $t$ before {or at the time at which} they exit the contemporary cluster. Similarly, define
\begin{equation}
    \label{eq:N-minus-def}
    N^-_{m,k} (t ; x) := \sum_{j=m+1}^{m+k} \1 { S^{(j)}_{\eta_{[-x,x]}^{(j)}} < -x, \, T^{(j)}_t \le \eta^{(j)}_{\Cl_{j-1}}} .
\end{equation}
We will always be assuming \eqref{ass:irreducible},
so that the random walk $S$ is irreducible, and hence
$\eta^{(j)}_{[-x,x]} < \infty$ for every $j$ and every $x$, almost surely. Note that  the sum $N^+_{m,k} (t ; x) +  N^-_{m,k} (t ; x)$ does not depend on~$x$ and hence we may define
\begin{equation}
    \label{eq:N-total-def}
    N_{m,k} (t ) := \sum_{j=m+1}^{m+k} \1 {  T^{(j)}_t \le \eta^{(j)}_{\Cl_{j-1}}}  = 
     N^+_{m,k} (t ; x) +  N^-_{m,k} (t ; x).
\end{equation}
\begin{remark}
\label{Rem:t_not_in_cluster_blg}
The definitions of $N_{m,k} (t)$, $N^{\pm}_{m,k} (t ; x)$ are motivated from similar quantities defined in \cite{lbg}, where the authors also make use of the property~\eqref{eq:P-not-in-Cl-2} below.
\end{remark}

Observe that, 
if there exists a (smallest) $j$ for which $T^{(j)}_t \leq  \eta^{(j)}_{\Cl_{j-1}}$,
then $t \in \Cl_{j}$. In particular,
\begin{equation}
\label{eq:P-not-in-Cl-2}
\{ t \in \Cl_{m+k} ,\,  t \notin \Cl_m  \} = \{ N_{m,k} (t) > 0 \}, \text{ and }
\Pr ( t \notin \Cl_{m+k} \setminus \Cl_m ) = \Pr ( N_{m,k} (t) = 0 ) .\end{equation}
It turns out to be convenient to work, instead of  with the quantities 
defined in~\eqref{eq:N-plus-def}--\eqref{eq:N-total-def}, with 
some related quantities, for the indefinitely extended walks,
that dispense with the dependence on the
contemporary cluster and so enjoy greater independence structure.
Define, for $s > 1$,
\begin{equation}
    \label{eq:K-plus-def}
    K^+_{m,k} (t, x, s) :=  
    \sum_{j=m+1}^{m+k} \1 { S^{(j)}_{\eta_{[-x,x]}^{(j)}} > x, \, T_t^{(j)} < \eta^{(j)}_{[- x,s x]}},
    \end{equation}
the number of the next $k$ random walks, released after time~$m$,
that (i) exit the interval $[-x,x]$ on the right and (ii) visit $t$ before they exit $[-x, s x]$.
The events in the indicators
in~\eqref{eq:K-plus-def} are i.i.d., so that
\begin{equation}
\label{eq:bin-plus}
K^+_{m,k} (t, x, s)  \sim \Bin \left(k , p^+(t, x, s)\right) , 
\end{equation}  
where
\begin{equation}
    \label{eq:ptx-plus-def}
p^+ (t, x, s) := \bP_0 \bigl(  S_{\eta_{[-x,x]}} > x, \, T_t < \eta_{[-x,s x]} \bigr).
\end{equation}
Similarly, let 
\begin{equation}
    \label{eq:K-minus-def}
    K^-_{m,k} (t, x , s) :=  
    \sum_{j=m+1}^{m+k} \1 { S^{(j)}_{\eta_{[-x,x]}^{(j)}} <-x , \, T_t^{(j)} < \eta^{(j)}_{[-s x, x]}},
    \end{equation}
be the number of the next $k$ random walks, released after time~$m$,
that (i) exit the interval $[-x,x]$ on the left and (ii) visit $t$ before they exit $[-s x, x]$.
The events in the indicators
in~\eqref{eq:K-minus-def} are i.i.d., so that
\begin{equation}
\label{eq:bin-minus}
K^-_{m,k} (t, x , s)  \sim \Bin \left(k , p^-(t, x, s)\right) , 
\end{equation}
where
\begin{equation}
    \label{eq:ptx-minus-def}
p^- (t , x , s) := \bP_0 \bigl(  S_{\eta_{[-x,x]}} <- x, \, T_t < \eta_{[-s x,x]} \bigr).
\end{equation}

The following is our basic tool for obtaining upper bounds on $\sigma_x$
in terms of the quantities $K^{\pm}_{\sigma_x,k} (t,x,s)$ as defined in equations~\eqref{eq:K-plus-def} and~\eqref{eq:K-minus-def}.

\begin{lemma}
\label{lem:over-filling}
Fix $s > u > 1$.
If, for $x,k \in \N$,
 it holds that    
\begin{equation}
\label{eq:K-large}
\min_{t \in \Z \cap (x,ux]} \min \left( K^+_{\sigma_x,k}(t,x,s),  K^-_{\sigma_x,k}(-t,x,s) \right) \geq \lceil (s-1)x \rceil,
\end{equation}
then $\sigma_{ux} \leq \sigma_x + k$. 
\end{lemma}
\begin{proof}
Consider $t \in \Z \cap (x,ux]$. If $K^+_{\sigma_x,k}(t,x,s)\geq \lceil (s-1)x \rceil$,
then at least $\lceil (s-1)x \rceil $ (indefinitely extended) 
random walks, released after time~$\sigma_x$, visit site~$t$ before exiting $[-x,sx]$.
The set $[-x,sx] \setminus [-x,x]$ contains no more
than $(s-1)x$ sites of $\Z$,
and so not all of the $\lceil (s-1)x \rceil$ particles can terminate before reaching~$t$.
Hence \[ K^+_{\sigma_x,k}(t,x,s)\geq \lceil (s-1)x \rceil \text{ implies that }
t \in \Cl_{\sigma_x +k}.
\]
Similarly, 
$ K^-_{\sigma_x,k}(-t,x,s)\geq \lceil (s-1)x \rceil$  implies that 
$-t \in \Cl_{\sigma_x +k}$.
It follows that~\eqref{eq:K-large} implies that  $t \in \Cl_{\sigma_x +k}$ for all $t \in [-ux, ux] \cap \Z$, and hence
$\sigma_{ux} \leq \sigma_x + k $.
\end{proof}

\subsection{Finite variance}
\label{sec:inner-radius-finite-variance}

To apply Lemma~\ref{lem:over-filling},
we need to choose $k = k_x$, depending on~$x$,
in such a way that the event in~\eqref{eq:K-large} 
occurs with high probability. To do so,
we will use binomial concentration
applied to~\eqref{eq:bin-plus} and~\eqref{eq:bin-minus}, and hence
we need to quantify the asymptotics of $p^+$ and $p^-$
defined in~\eqref{eq:ptx-plus-def} and~\eqref{eq:ptx-minus-def}, respectively.
We achieve this through results which  make the heuristic ideas in~\S\ref{sec:inner-radius-heuristics} formal, and will enable us to prove 
Proposition~\ref{prop:inner-radius-inverse} and hence Theorem~\ref{thm:idla-light-tail}. {First, the next result is a formal statement of the heuristic in point 2a in the outline in~\S\ref{sec:inner-radius-heuristics} derived from gambler's ruin estimates and tightness of the overshoots.}

\begin{lemma}
\label{lem:t_visit_good_prob}
Suppose that~\eqref{ass:irreducible} holds,  $\Exp ( X^2) < \infty$, and $\Exp X =0$.
Then for every $\eps >0$, there is a $u_\eps > 1$ such that the following holds. For
every $u \in (1, u_\eps)$ and every $s > u$,
\begin{equation}
\label{eq:p-plus-minus-eps-bound}
 \lim_{x \to \infty}       \sup_{t \in \Z \cap [x,ux]}  \left| p^\pm(t,x,s)-\frac{1}{2} \right| <\varepsilon.
    \end{equation}
    \end{lemma}
\begin{proof}
First observe that, by the gambler's ruin asymptotics in
Lemma~\ref{lem:kesten-gambler}\ref{lem:kesten-gambler-i},
\begin{equation}
\label{eq:exit-balance}
\lim_{x \to \infty}\bP_0 \bigl( S_{\eta_{[-x,x]}} > x\bigr) = 1/2.
\end{equation}
Consequently, for every $s >1$, we have the upper bound 
\begin{equation}
    \label{eq:p-plus-upper}
 \limsup_{x \to \infty} \sup_{t \in \Z} \bP_0 \bigl(  S_{\eta_{[-x,x]}} > x, \, T_t < \eta_{[- x,s x]} \bigr) \leq \frac{1}{2} .
\end{equation}
To obtain a lower bound, consider $t \in \Z \cap[x,ux]$, and write, for $B \in \RP$,
\begin{align*}
\bP_0 \bigl(  S_{\eta_{[-x,x]}} > x, \, T_t < \eta_{[- x,s x]} \bigr)
& \geq \bE_0 \Bigl[ \bP_0 \bigl(  T_t < \eta_{[- x,s x]} \bigmid \cF_{\eta_{[-x,t]}} \bigr) \1{ t < S_{\eta_{[-x,t]}} < t +B } \Bigr],
    \end{align*}
    where, as in \S\ref{sec:kesten},
    $\cF_n = \sigma (S_0,S_1, \ldots, S_n)$, with respect to which $\eta_{[-x,t]}$ is a stopping time. 
    By Lemma~\ref{prop:walk-overshoot}\ref{prop:walk-overshoot-ii} (tightness of the overshoots)
we have that, for every $\eps>0$,
we may choose $B$ large enough so that $\sup_{t \in \Z} \bP_0 ( S_{\rho_t} \geq t + B ) \leq \eps$;
fix $\eps >0$ and such a $B$.
    By the strong Markov property, 
    it holds that, on $\{  t < S_{\eta_{[-x,t]}} < t +B \}$,
\begin{align*}
\min_{t \in [x,ux]} \bP_0 \bigl(  T_t < \eta_{[- x,s x]} \bigmid \cF_{\eta_{[-x,t]}} \bigr)
& \geq \min_{t \in [x,ux]}  \min_{y \in [t,t+B] \cap \Z} \bP_y  \bigl(  T_t < \eta_{[- x,s x]} \bigr)\\
& \geq \min_{t \in [x,ux]} \min_{y \in [0,B] \cap \Z} \bP_y  \bigl(  T_0 < \eta_{[-x-t ,sx-t]} \bigr) \\
& \geq \min_{y \in [0,B] \cap \Z} \bP_y  \bigl(  T_0 < \eta_{[-x,(s-u) x]} \bigr) ,\end{align*}
since $\eta_{[-x,(s-u)x]} \leq \eta_{[-x-t,sx-t]}$ for every $t \in [x,ux]$. This last bound
tends to~$1$
as $x \to \infty$, provided $s > u >1$,
by the local hitting property, Lemma~\ref{lem:recurrence-hitting}. 
Hence for every $\eps>0$,
and every $s > u >1$, for all $x$ large enough 
\begin{align*}
\min_{t \in [x,ux]}  \bP_0 \bigl(  S_{\eta_{[-x,x]}} > x, \, T_t < \eta_{[- x,s x]} \bigr)
& \geq \min_{t \in [x,ux]}  \bP_0 \bigl( t < S_{\eta_{[-x,t]}} < t +B \bigr) -\eps.
    \end{align*}
    Recalling from~\eqref{eq:rho_x-def} and~\eqref{eq:T-eta-S} that $\rho_t = \eta_{(-\infty, t]}$, we have 
\begin{align*}
\bP_0 \bigl( t < S_{\eta_{[-x,t]}} < t +B \bigr) & = 
\bP_0 \bigl(  S_{\eta_{[-x,t]}} > t , \, S_{\rho_t} < t + B \bigr) \\
& \geq     \bP_0 \bigl(  S_{\eta_{[-x,t]}} > t \bigr) - \bP_0 ( S_{\rho_t} \geq t + B ) \\
& \geq     \bP_0 \bigl(  S_{\eta_{[-x,t]}} > t \bigr) - \eps, 
\end{align*}
by choice of~$B$. 
Since $\{ S_{\eta_{[-x,t+1]}} \geq t+1 \} \subseteq \{ S_{\eta_{[-x,t]}} \geq t \}$ for every $t \in \ZP$, 
 another application of the gambler's ruin asymptotics in
Lemma~\ref{lem:kesten-gambler}\ref{lem:kesten-gambler-i} shows that, 
for every $\eps >0$ and $u >1$, for all $x$ large enough,
\[ 
\min_{t \in [x,ux]} \bP_0 \bigl(  S_{\eta_{[-x,t]}} > t \bigr) 
\geq 
\bP_0 \bigl(  S_{\eta_{[-x,ux]}} > ux \bigr) \geq \frac{1}{1+u} - \eps.
\]
Thus we conclude that for every $\eps >0$ and $s > u >1$, for all $x$ large enough,
\[ 
\min_{t \in [x,ux]}  \bP_0 \bigl(  S_{\eta_{[-x,x]}} > x, \, T_t < \eta_{[- x,s x]} \bigr)
\geq  \frac{1}{1+u} - \eps.
\]
In particular, we can choose $u>1$ close enough to~$1$ so that this last probability is arbitrarily close
to $1/2$, which combines with~\eqref{eq:p-plus-upper} to conclude the proof of the statement for $p^+$ in~\eqref{eq:p-plus-minus-eps-bound}.
The proof of the statement for $p^-$ is analogous.
\end{proof}

The following 
high-probability statements
are obtained from the probability bounds in Lemma~\ref{lem:t_visit_good_prob}
via binomial concentration, {giving a formal version
of the finite-variance case of the heuristic in point 3 in the outline in~\S\ref{sec:inner-radius-heuristics}}.

\begin{lemma}
\label{lem:single-concentration}
Suppose that~\eqref{ass:irreducible} holds,   $\Exp ( X^2) < \infty$, and $\Exp X =0$. There exists $\eps_0 \in (0,1)$ such that the following holds. 
Take $\eps \in (0,\eps_0)$, and let $u_\eps > 1$ be as given in Lemma~\ref{lem:t_visit_good_prob}.
Then, for every $u \in (1, u_\eps)$ and every $s > u$,
there exist $c>0$ (depending on $\eps$) and $x_0 >0$ such that, for all $x>x_0$, 
with $k_x: = \lceil (2 + 7\eps) (s-1)x \rceil$, 
\begin{equation}
\label{eq:concentration-single-t-plus}
    \max_{t \in \Z \cap [ x,ux]} \Pr\left(K^+_{\sigma_x,k_x}(t,x,s)<(s-1)x\right)  \leq \re^{-c(s-1)x},
    \end{equation}
    and 
    \begin{equation}\label{eq:concentration-single-t-minus}
        \max_{t \in \Z \cap [ x,ux]} \Pr \left( K^-_{\sigma_x,k_x}(-t,x,s)< (s-1)x\right)\leq \re^{-c(s-1)x}. 
    \end{equation}
\end{lemma}
\begin{proof}
Standard Chernoff bounds for binomial large deviations
(see e.g.~\cite[p.~16]{penrose}) say that if $X \sim \Bin (n,p)$ and $a \in (0,1)$, then $\Pr ( X < a np ) \leq \exp ( - c_a np )$, where $c_a >0$. 
    Recall from~\eqref{eq:bin-plus} that $K^+_{\sigma_x,k}(t,x,s)$ follows a binomial distribution with mean $k p^+(t,x,s)$. Choose
    $k_x = \lceil (2 + 7\eps) (s-1)x \rceil$. 
    For $u \in (1,u_\eps)$, $s>u$, and $x$ large enough,
     Lemma~\ref{lem:t_visit_good_prob} shows that, for all $t \in \Z \cap [ x,ux]$,
     \[ k_x p^+(t,x,s) \geq  (2+7\eps) (s-1)x \cdot \left( \frac{1}{2} - \eps \right)
     > (1+ \eps) (s-1)x ,\]
     for all $\eps \in (0,\eps_0)$ sufficiently small. The binomial Chernoff bound stated above then yields~\eqref{eq:concentration-single-t-plus},
     where $c >0$ depends only on $\eps$. A similar argument yields~\eqref{eq:concentration-single-t-minus}.
    \end{proof}

    {Now we can complete the proof of Proposition~\ref{prop:inner-radius-inverse}.
    To do so, we combine Lemmas~\ref{lem:over-filling} and~\ref{lem:single-concentration}
    to achieve the formal version of point 5 in the outline in~\S\ref{sec:inner-radius-heuristics}
    to give an upper bound as previewed in equation~\eqref{eq:outline-bound} with $q \approx 1/2$ and $s \approx u \approx 1$ chosen appropriately.}

\begin{proof}[Proof of Proposition~\ref{prop:inner-radius-inverse}]
Let $\eps_0 >0$ be the constant from Lemma~\ref{lem:single-concentration}.
Take $\eps \in (0,\eps_0)$, and let $u_\eps > 1$ be as given in Lemma~\ref{lem:t_visit_good_prob}.
It follows from Lemma~\ref{lem:single-concentration} that,
for every $u \in (1, u_\eps)$ and every $s > u$,
there exist $c>0$ (depending on $\eps$) and $x_0 >0$ such that, for all $x>x_0$, 
with $k_x = \lceil (2 + 7\eps) (s-1)x \rceil$, 
\begin{align*}
& {}    \Pr \left( \min_{t \in \Z \cap (x,ux]} \min \left( K^+_{\sigma_x,k_x}(t,x,s),  K^-_{\sigma_x,k_x}(-t,x,s) \right) < \lceil (s-1)x \rceil \right) \\
& {} \quad {} \leq 2 u x \max_{t \in \Z \cap (x,ux]} \max \left\{  \Pr\left(K^+_{\sigma_x,k_x}(t,x,s)<(s-1)x\right) , \Pr\left( K^-_{\sigma_x,k_x}(-t,x,s)< (s-1)x \right) \right\} \\
    & {} \quad {} \leq 2 u x \re^{-c(s-1)x}.
\end{align*}
In particular, we obtain
\[ \sum_{x \in \N} 
 \Pr \left( \min_{t \in \Z \cap (x,ux]} \min \left( K^+_{\sigma_x,k_x}(t,x,s),  K^-_{\sigma_x,k_x}(-t,x,s) \right) < \lceil (s-1)x \rceil \right) < \infty.
\]
It follows from Lemma~\ref{lem:over-filling} and the 
Borel--Cantelli lemma that, almost surely, for all but finitely many $x \in \N$, \eqref{eq:K-large} holds and thus 
$\sigma_{ux} \leq \sigma_x + k_x$.
In particular, considering the subsequence $x = u^m$, it follows that there is
a (random, a.s.~finite) $m_0 \in \N$ such that
\[
\sigma_{u^{m+1}} - \sigma_{u^m}  \leq k_{u^m} 
\leq 1+ (2 + 7\eps) (s-1) u^m, \text{ for all } m \geq m_0.
\]
Consequently, for all $m \geq m_0$,
\begin{align}
    \sigma_{u^m} & = \sigma_{u^{m_0}} + \sum_{\ell = m_0}^{m-1} \left( \sigma_{u^{\ell+1}}-\sigma_{u^\ell} \right) \nonumber\\
    & \leq  \sigma_{u^{m_0}} + m+ (2 + 7\eps) (s-1) \sum_{\ell =0}^{m-1}  u^\ell \nonumber\\
    \label{eq:sigma_x-growth-bound}
&\leq  \sigma_{u^{m_0}} + m+ (2 + 7\eps) \left( \frac{s-1}{u-1} \right) u^{m}.
\end{align}
For fixed $u \in (1,u_\eps)$ and every $x \in \N$, there exists $m_x \in \ZP$ such that $u^{m_x} \leq x < u^{m_x + 1}$;
note that $m_x \to \infty$ as $x \to \infty$.
Since $\sigma_x \leq \sigma_{u^{m_x+1}}$, it follows from \eqref{eq:sigma_x-growth-bound} that, 
for each  $u \in (1,u_\eps)$ and $s >u$,
\begin{align}
\label{eq:s-u-algebra}
\limsup_{x \to \infty} 
\frac{\sigma_x}{x} & \leq \limsup_{x \to \infty}\frac{\sigma_{u^{m_x + 1}}}{u^{m_x}} \nonumber\\
& \leq (2 + 7\eps) \left( \frac{s -1}{u-1} \right) u +  \limsup_{x \to \infty} \frac{\sigma_{u^{m_0}} + m_x+1}{u^{m_x}}  \nonumber\\
& =  (2 + 7\eps) \left( \frac{s -1}{u-1} \right) u.
\end{align}    
Since, for fixed $\eps \in (0,\eps_0)$, the choices of $s$ and $u$ were arbitrary subject to $s > u$ and $u \in (1,u_\eps)$, it follows
from~\eqref{eq:s-u-algebra} that 
$\limsup_{x \to \infty} (\sigma_x / x) \leq 2 + 7\eps$, a.s.
Since $\eps \in (0,\eps_0)$ was arbitrary, this completes the proof.
 \end{proof}

\begin{proof}[Proof of Theorem~\ref{thm:idla-light-tail}]
As explained in \S\ref{sec:inner-radius-heuristics}, Theorem~\ref{thm:idla-light-tail}
follows from the 
bound~\eqref{eq:r-R-trivial} together with the bound from Proposition~\ref{prop:inner-radius-inverse} and the inversion described by~\eqref{eq:inversion}.
\end{proof}

\subsection{Infinite variance}
\label{sec:inner-radius-infinite-variance-upper-bound}

{In the infinite-variance case, 
the present section deals with the proof of the upper bound on $\sigma_x/x$ given in
Proposition~\ref{prop:upper-bound_infinite-variance}.} For Theorem~\ref{thm:idla-heavy-tail}
we also need a \emph{lower} bound $\sigma_x/x >c$ with $c>2$, which we establish in \S\ref{sec:inner-radius-infinite-variance-lower-bound} below. 

The structure of this section parallels that of \S\ref{sec:inner-radius-finite-variance}. The strategy is to once more apply
 Lemma~\ref{lem:over-filling},
but now  the asymptotics of $p^+$ and $p^-$
defined in~\eqref{eq:ptx-plus-def} and~\eqref{eq:ptx-minus-def}, respectively,
are different. 
Define
\begin{equation}
    \label{eq:q-lower-bound}
    \uq (u , s) := \left( \frac{\alpha-1}{1+u} \right) \left( \frac{s-u}{s} \right)^{1-\frac{\alpha}{2}} . 
    \end{equation}
 {The following result takes the place of Lemma~\ref{lem:t_visit_good_prob} in the infinite-variance setting, to provide a formal statement of the heuristic in point 2b in the outline in~\S\ref{sec:inner-radius-heuristics}. Again, gambler's ruin estimates mean that there is a roughly $1/2$ chance that a walk exits $[-x,x]$ on the right, but now the lack of tightness of the overshoot means that we cannot guarantee to hit a nearby point before going far away; instead we use Kesten's hitting-before-exit  estimates from~\S\ref{sec:kesten} to obtain the following lower bounds.}

\begin{lemma}
\label{lem:t_visit_good_prob_infinite_variance}
Suppose that~\eqref{ass:irreducible} holds, 
and that~\eqref{ass:stable-alpha} holds with $\alpha \in (1,2)$.
For
 every $u > 1$ and every $s > u$,
with $\uq$ as defined in equation~\eqref{eq:q-lower-bound}, it holds that
\begin{equation}
\label{eq:p-plus-minus-eps-bound_infinite_variance}
   \liminf_{x \to \infty}     \min_{t \in \Z \cap [x,ux]}   p^\pm (t,x,s) \geq  
   \uq (u,s) .
    \end{equation}
    \end{lemma}

In what follows, as in \S\ref{sec:inner-radius-finite-variance}, we are free to choose $s$ and $u$ such that $s > u > 1$, and,
as before, for a fixed probability lower bound on $p^\pm$, the optimal choice would be to take $u \approx 1$ and $s \approx u$. However, the bound in~\eqref{eq:p-plus-minus-eps-bound_infinite_variance} has $\uq (u,s) \to 0$ as $s-u \to 0$,
and in fact 
this is inevitable in the infinite-variance case, since the overshoot is not tight. Thus we must keep $s -u$ strictly positive, and then (compare~\eqref{eq:s-u-algebra}) one must also keep $u -1$ strictly positive. The balance is then to choose $u, s$ so that $s-u$ and $u-1$ are positive, but not too large. The optimal choice can be found by some calculus, but provides a somewhat complicated formula. As our aim here is not to obtain the optimal constants, but to provide reasonable bounds that capture important asymptotics (such as $\alpha \uparrow 2$ behaviour) we instead will choose $u = 1 +h$, $s=u+h^2=1+h+h^2$. Then
    \begin{align}
    \label{eq:uq-choice}
        \uq (1+h , 1+h+h^2) = \frac{(\alpha -1)}{(2 + h)} \frac{h^{2-\alpha}}{(1 + h + h^2)^{1-\alpha/2}}. 
    \end{align}
    The bound~\eqref{eq:uq-choice} has the property that it goes to $1/2$ as $h \downarrow 0$ and $\alpha \uparrow 2$ appropriately.
    
    \begin{proof}[Proof of Lemma~\ref{lem:t_visit_good_prob_infinite_variance}]
        Note that for $t \in \Z \cap [x,ux]$ we have 
        \begin{align*} p^+(t,x,s) &
        = \Pr_0 (T_t < \eta_{[-x,sx]} )
        = \Pr_{-t} (T_0 < \eta_{[-x-t,sx-t]} )  = \Pr_{t} (T_0 < \eta_{[t-sx,x+t]} ),\end{align*}
        by a change of sign and the symmetry hypothesis in~\eqref{ass:stable-alpha}.  
        Hence, using the definition of  \eqref{Eq: def_q_N}, we obtain
        \[ p^+(t,x,s)  = q_{\alpha,x+t} \left( \frac{t}{x+t} ; \frac{sx-t}{x+t} \right) .\]
Note that for $ t \in [x,ux] \cap \Z$ we have $\frac{sx-t}{x+t} \geq \frac{s-u}{1+u}$ and thus, by the monotonicity property \eqref{eq:q-monotone} we obtain that
\[
q_{\alpha,x+t} \left( \frac{t}{x+t} ; \frac{sx-t}{x+t} \right) \geq q_{\alpha,x+t} \left( \frac{t}{x+t} ; \frac{s-u}{1+u} \right).
\]

It follows  from \eqref{eq:q-limit-y-limit} and \eqref{eq:q-inf-convergence}, which are consequences of
Kesten's convergence and equicontinuity results as presented in Lemmas~\ref{lem:kesten-visit-before-exit} and~\ref{lem:kesten-equicontinuity} that
\[ 
\lim_{x \to \infty} \min_{t \in \Z \cap [x,ux]}
q_{\alpha,x+t} \left( \frac{t}{x+t} ; \frac{s-u}{1+u} \right)
= \inf_{y \in [ \frac{1}{2}, \frac{u}{1+u}]} q_\alpha \left(y ; \frac{s-u}{1+u} \right)
.
\]
Using the lower bound from equation~\eqref{eq:lower_bound_q} {and observing that $\frac{\alpha}{2}-1<0$} we get
\begin{align*}
\inf_{y \in [ \frac{1}{2}, \frac{u}{1+u}]} q_\alpha \left(y ; \frac{s-u}{1+u} \right)
& \geq (\alpha-1) \left( \frac{s-u}{1+u} \right)^{1-\frac{\alpha}{2}} 
\inf_{y \in [ \frac{1}{2}, \frac{u}{1+u}]} \left[ \left( y + \frac{s-u}{1+u}\right)^{\frac{\alpha}{2}-1}(1-y)  \right] \\ 
& = (\alpha-1) \left( \frac{s-u}{1+u} \right)^{1-\frac{\alpha}{2}} 
\left(  \frac{s}{1+u}\right)^{\frac{\alpha}{2}-1} \frac{1}{1+u},
\end{align*}
which is equal to $\uq (u,s)$ as defined in equation~\eqref{eq:q-lower-bound}.
{The same result holds for $p^-$ by the symmetry hypothesis~\eqref{ass:stable-alpha}.}
    \end{proof}

The next result will substitute for Lemma~\ref{lem:single-concentration}
in the infinite-variance setting, 
 {giving a formal version
of the infinite-variance case of the heuristic in point 3 in the outline in~\S\ref{sec:inner-radius-heuristics}. The proof again uses binomial concentration with the distributional equalities in~\eqref{eq:bin-plus} and~\eqref{eq:bin-minus}, this time with the probability
estimates from Lemma~\ref{lem:t_visit_good_prob_infinite_variance}}.

\begin{lemma}
\label{lem:single-concentration-stable}
Suppose that~\eqref{ass:irreducible} holds and that~\eqref{ass:stable-alpha} holds with $\alpha \in (1,2)$. Take $u,s \in \R$, such that $1<u<s$ and let $q: = \uq(u,s)$ as given in equation~\eqref{eq:q-lower-bound}.
Then, there exist $\eps_0 \in (0,1)$, $K>0$, and $x_0 >0$ such that the following holds. 
For all $x>x_0$, every $\eps \in (0,\eps_0)$, 
with $k_x : = \lceil (q^{-1} + K\eps) (s-1)x \rceil$, there is $c >0$ for which
    \begin{equation}
    \label{eq:concentration-single-t-plus-stable}
    \max_{t \in \Z \cap [ x,ux]} \Pr\left(K^+_{\sigma_x,k_x}(t,x,s)<(s-1)x\right)  \leq \re^{-c(s-1)x},
    \end{equation}
    and 
    \begin{equation}\label{eq:concentration-single-t-minus-stable}
        \max_{t \in \Z \cap [ x,ux]} \Pr \left( K^-_{\sigma_x,k_x}(-t,x,s)< (s-1)x\right)\leq \re^{-c(s-1)x}. 
    \end{equation}
\end{lemma}
\begin{proof}
Given $u,s$ with $1<u<s$, take $q := \uq(u,s)>0$ where $\uq (u,s)$
is defined in equation~\eqref{eq:q-lower-bound}. Then there exist $\eps_0>0$ and  $K \in \N$ (both depending on $q$ and hence on $u$ and $s$) such that 
\begin{equation}\label{eq:qK-control-eps}
\left(q^{-1} +K\eps\right) \cdot \left( q - \eps \right) > (1 + \eps) ,\text{ for all } \eps \in (0,\eps_0).
\end{equation}
Define $k_x = \lceil (q^{-1} + K\eps) (s-1)x \rceil$, as in the lemma, for this choice of~$K$.
Take $\eps \in (0,\eps_0)$.
By Lemma~\ref{lem:t_visit_good_prob_infinite_variance}, it holds that
$p^+(t,x,s) > \uq(u,s) - \eps$ for all $x$ large enough. Then it follows from the choice of $k_x$,
and property~\eqref{eq:qK-control-eps}, that
\begin{align*}
k_x p^+(t,x,s) &\geq  \left(q^{-1} +K\eps\right) (s-1)x \cdot \left( q - \eps \right)> (1+ \eps) (s-1)x.
\end{align*}
The binomial Chernoff bound stated in the proof of Lemma \ref{lem:single-concentration} then yields equation~\eqref{eq:concentration-single-t-plus-stable},
where $c >0$ depends only on $\eps$. A similar argument yields~\eqref{eq:concentration-single-t-minus-stable}.
\end{proof}

   {Now we can complete the proof of Proposition~\ref{prop:upper-bound_infinite-variance}.
   Similarly to the proof of Proposition~\ref{prop:inner-radius-inverse}, we combine Lemmas~\ref{lem:over-filling} and~\ref{lem:single-concentration-stable}
    to achieve the formal version of the infinite-variance case of point 5 in the outline in~\S\ref{sec:inner-radius-heuristics}
    to give an upper bound as previewed in equation~\eqref{eq:outline-bound} with $q = \uq (u,s)$ defined in equation~\eqref{eq:q-lower-bound} and $s > u > 1$ chosen appropriately as functions of $\alpha$
    as indicated in equation~\eqref{eq:uq-choice} above.}

\begin{proof}[Proof of Proposition~\ref{prop:upper-bound_infinite-variance}]
Let $1<u<s $ and $\eps_0>0$ be the constants from Lemma~\ref{lem:single-concentration-stable}.
It follows from Lemma~\ref{lem:single-concentration-stable} that for every $\eps \in (0,\eps_0)$
there exist $K>0$, depending on $q = \uq (u,s)$,  $c>0$ and $x_0 >0$, depending on $\eps$,  such that, for all $x>x_0$, 
with $k_x = \lceil (q^{-1} + K\eps) (s-1)x \rceil$, 
\begin{align*}
& {}    \Pr \left( \min_{t \in \Z \cap (x,ux]} \min \left\{ K^+_{\sigma_x,k_x}(t,x,s),  K^-_{\sigma_x,k_x}(-t,x,s) \right\} < \lceil (s-1)x \rceil \right) \\
    & {} \quad {} \leq u x \max_{t \in \Z \cap [x,ux]} \Big(  \Pr\left(K^+_{\sigma_x,k_x}(t,x,s)<(s-1)x\right) + \Pr\left( K^-_{\sigma_x,k_x}(-t,x,s)< (s-1)x \right) \Big) \\
    & {} \quad {} \leq 2 u x \re^{-c(s-1)x}.
\end{align*}
In particular, we obtain
\[ \sum_{x \in \N} 
 \Pr \left( \min_{t \in \Z \cap (x,ux]} \min \left( K^+_{\sigma_x,k_x}(t,x,s),  K^-_{\sigma_x,k_x}(-t,x,s) \right) < \lceil (s-1)x \rceil \right) < \infty.
\]
It follows from Lemma~\ref{lem:over-filling} and the 
Borel--Cantelli lemma that, almost surely, for all but finitely many $x \in \N$, equation~\eqref{eq:K-large} holds and thus 
$\sigma_{ux} \leq \sigma_x + k_x$. Proceeding as in the proof of Proposition~\ref{prop:inner-radius-inverse},
but with $k_x$ given in Lemma~\ref{lem:single-concentration} replaced by that in Lemma~\ref{lem:single-concentration-stable},
repeating the steps through equations~\eqref{eq:sigma_x-growth-bound}--\eqref{eq:s-u-algebra},
 we obtain 
\begin{align*}
\limsup_{x \to \infty} 
\frac{\sigma_x}{x} & \leq 
(q^{-1} + K\eps) \left( \frac{s -1}{u-1} \right) u, \as \end{align*}    
Here $\eps \in (0,\eps_0)$ was arbitrary, 
and $K, u, s$ do not depend on $\eps$,
so it holds that
\begin{equation}
    \label{eq:s-u-algebra-alpha}
\limsup_{x \to \infty} 
\frac{\sigma_x}{x} \leq  \frac{1}{\uq (u,s)} \left( \frac{s -1}{u-1} \right) u, \as 
\end{equation}
Then using the expression~\eqref{eq:uq-choice} 
together with equation~\eqref{eq:q-lower-bound}, it follows from~\eqref{eq:s-u-algebra-alpha} that
\begin{align}
\limsup_{x \to \infty} 
\frac{\sigma_x}{x} \leq  \frac{1+u}{\alpha -1} \left( \frac{s}{s-u} \right)^{1-\frac{\alpha}{2}} \left( \frac{s -1}{u-1} \right) u, \as \label{eq:limsup_su}
\end{align}    
Now choose $u = 1+h$ and $s=1+h+h^2$ for some $h \in [0,1]$. Then using the fact that $\alpha \in (1,2)$ and $h \in [0,1]$,
in the bound~\eqref{eq:limsup_su} we get, a.s.,
\begin{align*}
\limsup_{x \to \infty} \frac{\sigma_x}{x}
& \leq (\alpha -1)^{-1} (2+h)(1+h)^2 (1+h+h^2)^{1 -\alpha/2} h^{\alpha-2} \\
& \leq (\alpha -1)^{-1} (2+h)(1+h)^2 (1+h+h^2)^{1/2} h^{\alpha-2} \\
& \leq
(\alpha -1)^{-1} (2+h)(1+h)^2 (1+2h)^{1/2} h^{\alpha-2}  \\
& \leq
(\alpha -1)^{-1} (2+h)(1+h)^3 h^{\alpha-2}  .
\end{align*}
Now we take $h =2-\alpha$, to give
\[
\limsup_{x \to \infty} \frac{\sigma_x}{x} \leq (\alpha -1)^{-1} (4-\alpha)(3-\alpha)^3 (2-\alpha)^{\alpha-2}, \as
\]
This completes the proof, with the value for $C_\alpha'$ given in equation~\eqref{eq:C-alpha-dash}. \end{proof}

\subsection{Upper bounds: Infinite variance}
\label{sec:inner-radius-infinite-variance-lower-bound}

In this section we prove
lower bounds $\sigma_x/x >C'_\alpha$ with $C'_\alpha >2$. In particular, by the inversion argument relating $\sigma_x$ and $r_m$ presented in \eqref{eq:inversion}, we will establish the upper bound in Theorem~\ref{thm:idla-heavy-tail}.
Define $u_\alpha : (1,\infty) \to (0,1)$ by
\begin{equation}
\label{eq:u-alpha-def}
  u_\alpha (w) 
  := 2^{\frac{\alpha}{2}-1} (\alpha - 1)  \frac{ \sin (\pi \alpha/2)}{\pi}\int_{w-1}^{\infty} {\frac{\ud v}{v^{\alpha/2} (2+v)^{\alpha/2} (1+v)}}, \text{ for } w > 1.
\end{equation}

\begin{proposition}
\label{prop:inner-radius-infinite-variance-lower-i}Suppose that~\eqref{ass:irreducible} holds, 
and that~\eqref{ass:stable-alpha} holds with $\alpha \in (1,2)$.
With $C_\alpha'$ given by~\eqref{eq:C-alpha-dash}, set
\[ C_\alpha'' :=     2 + \sup \left\{  \frac{(C-2) \wedge u_\alpha ( (3/2) (C-1) )}{C+1} : C > C_\alpha' \right\}. \]
Then it holds that, a.s.,
\begin{equation}
\label{eq:sigma_liminf_infinite_variance-prop_lower}
\liminf_{x \to \infty} \frac{\sigma_x}{x} \geq C''_\alpha >2.
\end{equation}
\end{proposition}

A key component in the proof of 
Proposition~\ref{prop:inner-radius-infinite-variance-lower-i} is
Lemma~\ref{lem:d-l-estimate} below.
{Roughly speaking, the result says that between time $\sigma_x$ (at which point $[-x,x]$ is filled) and $\sigma_x +x$, of the $x$ additional random walkers, with very high probability at least $\theta x$ (for some $\theta>0$) particles end up outside $[-Ax , Ax]$ ($A>1$) due to taking a big jump to exit $[-x,x]$.
This is underpinned by  an extension due to Kesten~\cite{kesten1961a} of the Dynkin--Lamperti theorem  (Proposition~\ref{prop:dynkin-lamperti}), 
which shows that when each walk overshoots $[-x,x]$ there is a positive probability it ends up at a distance at least $Ax$. For appropriately large $A>1$ we can be confident that there is an unoccupied site nearby the exit location of these walks that take big jumps, and so, thanks to  Kesten's hitting estimates from Lemma~\ref{lem:kesten-visit-before-exit}, there is a good chance that the cluster adds a positive fraction of $x$ sites outside $[-Ax,Ax]$, and hence too far away to contribute to filling the cluster soon. Following this reasoning, and keeping track of constants, will lead to a proof of Proposition~\ref{prop:inner-radius-infinite-variance-lower-i}.}

Let $\cG_m  := \sigma( S^{(1)}, \ldots, S^{(m)} )$,
 the $\sigma$-algebra generated by the first $m$ random walks, 
and note that $\sigma_x$ defined in~\eqref{eq:sigma-x-def} is a stopping time with respect to the filtration $(\cG_m)_{m \in \ZP}$. 
For $m \in \ZP$, $A>0$ set 
\begin{equation}\label{eq:k_lost_particles} 
k_x(m,A) := \# \bigl( \Cl_{\sigma_x + m} \setminus [-Ax,Ax] \bigr),
\end{equation}
the number of occupied sites outside
$[-Ax,Ax]$ at time $\sigma_x + m$. For $A>1$, we think of $k_x (m,A)$
as ``lost particles'' that are too far away to contribute in the near future to the filling of the main cluster. The following result is a probability estimate that shows that lost particles are readily generated.

\begin{lemma}
    \label{lem:d-l-estimate}
Let $A>1$, $B>0$. Recall the definition of $u_\alpha$ from~\eqref{eq:u-alpha-def}. 
Then for every $u \in (0, u_\alpha(A+B))$ and every $\theta \in (0,B \wedge u)$,
there is an $x_0 \in \N$ such that, for all $x \geq x_0$,
\begin{equation}
    \label{eq:lost-particles-bound} \Pr ( k_x(x,A) \leq \theta x \mid \cG_{\sigma_x} )
\leq \exp ( - (\theta -u)^2 x /2), \text{ on } \{ k_x (0,A) < B x \}.
\end{equation}
\end{lemma}

\begin{proof}
For $A>0$, $x \in \ZP$, and $m \in \ZP$, let $\Delta_x (m,A): = k_x(m+1,A)- k_x(m,A)$ and note that $\Delta_x (m,A) = \1 { |S^{(\sigma_x+m+1)}_{\tau_{\sigma_x+m+1}}| > Ax }$. 
The first step in the proof of \eqref{eq:lost-particles-bound} is to obtain a lower bound on
\begin{align}
\label{eq:k-inc-prob}
\Pr( \Delta_x (m, A) = 1 \mid \cG_{\sigma_x+m} ) 
& =\Pr\bigl( \big|S^{(\sigma_x+m+1)}_{\tau_{\sigma_x+m+1}}\big| > Ax \bigmid \cG_{\sigma_x+m} \bigr).
\end{align} 
Let $B >0$. 
For $y \in \Z$, $x \in \ZP$, and $j > \sigma_x$, define 
\begin{equation}\label{notation-interval-empty}
    I_B(x,y,j): = \Bigl( \bigl[y - \tfrac{Bx}{2}, y + \tfrac{Bx}{2} \bigr]\cap \Z \Bigr) \setminus\Cl_{j},
\end{equation}  
the sites within distance $Bx/2$ of $y$ which are not occupied by the cluster at time~$j$.
Recall, from~\eqref{eq:cluster expansion} and \eqref{eq:eta-j-def}, that $\tau_j = \eta_{\Cl_{j-1}}^{(j)}$ and thus
$\Cl_j$ grows from $\Cl_{j-1}$ by addition of $S^{(j)}_{\tau_{j}}$,
that is,
$S^{(j)}_{\tau_{j}}$ is 
the first site outside of $\Cl_{j-1}$ visited by the $j$th random walk. Recall the definitions~\eqref{eq:T-j-def}--\eqref{eq:eta-j-def}, and define $\teta_B^{(j)}(x,y):=\inf\{n > \eta^{(j)}_{[-x,x]} :  S^{(j)}_n  \notin [y - Bx, y + Bx]\}$. With this notation $\teta_B^{(j)}(x,y)$ denotes the first time  the $j$th random walk lands outside the interval $[y - Bx, y + Bx]$ after it exits the interval $[-x,x]$.
Note that for $j > \sigma_x$, the $j$th random walk is still active when it exits $[-x,x]$. 
If $t \in I_B(x,y,j)$ and $T^{(j)}_t \leq  \teta_B^{(j)}(x,y)$, then 
$S^{(j)}_{\tau_{j}} \in [y - Bx,y + Bx]$ and $|S^{(j)}_{\tau_{j}}| >Ax $ for any $|y|>A+B$.
 Therefore, for any $j > \sigma_x$, we have 
\begin{equation}\label{eq:land-absorb}
\bigcup_{y \in \Z :\, |y| > (A+B) x} \Big\{ 
S^{(j)}_{\eta^{(j)}_{[-x,x]}} = y, \, \min_{t \in I_B(x,y,j)}  T^{(j)}_t < \teta_B^{(j)}(x,y)  \Big\} \subseteq \Big\{\big|S^{(j)}_{\tau_j} \big|> Ax\Big\} .
\end{equation}
Now, for $I \subset \Z$, set 
\begin{align*}
    F_B  (x,y, I ) := \bP_0 \Bigl( S_{\eta_{[-x,x]}} = y, \, \min_{t \in I}  T_t < \teta_B (x,y) \Bigr).
\end{align*}
By the strong Markov property applied at time $\eta_{[-x,x]}$, we have
\begin{align*}
    F_B (x,y, I ) & = \bP_0 \bigl(  S_{\eta_{[-x,x]}} = y \bigr) \bP_y \Bigl(  \min_{t \in I}  T_t < \eta_{[y-Bx,y+Bx]} \Bigr)\\
    & \geq \bP_0 \bigl(  S_{\eta_{[-x,x]}} = y \bigr) \inf_{t \in I} \bP_y \bigl(  T_{t} < \eta_{[y-Bx,y+Bx]} \bigr). 
\end{align*}
Combined with~\eqref{eq:land-absorb}, it follows that, for $j > \sigma_x$, on $\{ I_B(x,y,j) \neq \emptyset \}$,
\begin{align}
\label{eq:two-part-bound}
  &  \Pr\bigl( \big|S^{(j+1)}_{\tau_{j+1}}\big| \geq Ax \bigmid \cG_{j} \bigr) \nonumber\\
& {} \quad{} \geq\sum_{y \in \Z: \, | y| > (A+B)x} F_B(x,y, I_B(x,y,j) ) \nonumber\\
& {} \quad{}  \geq \sum_{y \in \Z: \, | y| > (A+B)x} \bP_0 \bigl(  S_{\eta_{[-x,x]}} = y  \bigr)  \inf_{t \in [y- \frac{Bx}{2}, y+\frac{Bx}{2} ] \cap \Z} \bP_y \bigl( T_t <\eta_{[y-Bx, y+Bx]}\} \bigr) . 
\end{align}
Now, by the symmetry assumption of the increments contained in~\eqref{ass:stable-alpha}, we have 
\begin{equation*}
\inf_{t \in [y- \frac{Bx}{2}, y+\frac{Bx}{2} ]\, \cap \Z} \bP_y \bigl( T_t <\eta_{[y-Bx, y+Bx]}\} \bigr)=\inf_{t \in [y-\frac{Bx}{2}, y]\, \cap \Z} \bP_y \bigl( T_t <\eta_{[y-Bx, y+Bx]}\} \bigr).
\end{equation*}
By translation invariance,
\[
\bP_y \bigl( T_t <\eta_{[y-Bx, y+Bx]} \bigr)= \bP_{y-t} \bigl( T_0 <\eta_{[y-t-Bx, y-t+Bx]}\} \bigr),
\]
 for all $y$ and $t$. 
Therefore, we have 
\begin{equation*}
\inf_{t \in [y-\frac{Bx}{2}, y] \, \cap \Z} \bP_y \bigl( T_t <\eta_{[y-Bx, y+Bx]}\} \bigr)= \inf_{z \in [0, \frac{Bx}{2}] \, \cap \Z} \bP_z \bigl( T_0 <\eta_{[z-Bx, z+Bx]}\} \bigr).
\end{equation*}
Since $\bP_{z} (T_0 < \eta_{[z-Bx,z+ Bx]}) \geq \bP_z  (T_0 < \eta_{[-\frac{Bx}{2},Bx ]})$ for $z \in [0,\frac{Bx}{2}] \cap \Z$, it follows that
\begin{equation*}
\inf_{z \in [0, \frac{Bx}{2}] \, \cap \Z} \bP_z \bigl( T_0 <\eta_{[z-Bx, z+Bx]} \bigr) 
\geq \inf_{z \in [0, \frac{Bx}{2}] \, \cap \Z} \bP_z \bigl( T_0 <\eta_{[-\frac{Bx}{2}, Bx]} \bigr).
\end{equation*}
Using the notation in \eqref{Eq: def_q_N} we may write
\begin{align*}
\inf_{z \in [0, \frac{Bx}{2}] \, \cap \Z} \bP_z \bigl( T_0 <\eta_{[-\frac{Bx}{2}, Bx]} \bigr) 
& = \inf_{\frac{z}{Bx} \in [0, \frac{1}{2}],\,  z \in \Z} 
q_{\alpha, Bx} \bigl( \tfrac{z}{Bx}; 1/2 \bigr)& \ge \inf_{y \in [0, \frac{1}{2}]} 
q_{\alpha, Bx} \bigl( y ; 1/2 \bigr).
\end{align*}
From the equicontinuity of $q_{\alpha,x}$
as in equation~\eqref{eq:q-inf-convergence} and the lower bound
from~\eqref{eq:lower_bound_q} we obtain
\begin{align} 
 \lim_{x \to \infty} \inf_{y \in [0, \frac{1}{2}]} 
q_{\alpha, Bx} \bigl( y ; 1/2\bigr) 
   & = \inf_{0 \leq y \leq 1/2} q_\alpha ( y; 1/2 ) 
    \geq (\alpha - 1) (1/2)^{2 - \frac{\alpha}{2}}  =: r_\alpha.
    \label{eq:good_absorption_probability_away} 
\end{align}

Let $\eps>0$. Then, using the bound~\eqref{eq:good_absorption_probability_away} in equation~\eqref{eq:two-part-bound}, we obtain, for all $x \geq x_0$ large enough, 
 on the event $\bigcap_{y: |y|>(A+B)x}\{ I (x,y,j ) \neq \emptyset \}$, 
\begin{align}
\label{eq:two-part-bound-2}
\Pr \bigl(\big|S^{(j+1)}_{\tau_{j+1}}\big| > A x \bigmid \cG_{j} \bigr)  
& \geq ( r_\alpha - \eps) \sum_{y \in \Z: \, | y| > (A+B)x} \bP_0 \bigl(  S_{\eta_{[-x,x]}} = y  \bigr) \nonumber\\
& = (r_\alpha -\eps) \bP_0 \bigl( \bigl|S_{\eta_{[-x,x]}} \bigr| > (A+B)x \bigr).
\end{align}
Kesten~\cite[p.~270]{kesten1961b}
considers the quantity $G_\alpha ( w ; x, 1):= \bP_0 ( x < S_{\eta_{[-x,x]}} \leq (w+1) x )$.
Under hypothesis~\eqref{ass:stable-alpha}, Kesten 
proved in~\cite{kesten1961a} that $\lim_{x \to \infty} G_\alpha (w; x, 1) = G_\alpha (w,1)$ exists, and gives a formula for $G_\alpha (w,1)$ in Theorem 1 of~\cite[p.~271]{kesten1961b}.
Moreover, the symmetry  assumed in~\eqref{ass:stable-alpha}
shows that $\bP_0 (|S_{\eta_{[-x,x]}}| > (w+1)x) = 2 \bP_0 ( S_{\eta_{[-x,x]}} > (w+1)x)$.
In particular, it follows from Theorem 1 of~\cite{kesten1961b} that, for $w >0$,
\begin{equation}
\label{eq:prob_large_overshoot}
\lim_{x \to \infty}\bP_0 \bigl( \bigl|S_{\eta_{[-x,x]}} \bigr| > (w+1)x \bigr)
=  \frac{2 \sin (\pi \alpha/2)}{\pi}\int_{w}^{\infty} {\frac{\ud v}{v^{\alpha/2} (2+v)^{\alpha/2} (1+v)}}
=: s_\alpha (w).
\end{equation} 
Combining equations~\eqref{eq:two-part-bound-2} and \eqref{eq:prob_large_overshoot},  we obtain for $\delta \in (0,1/2)$ that, for every $\eps >0$, for all $x \geq x_0$ large enough, on $\{ I (x,y,j) \neq \emptyset \}$,
\begin{align}
\label{eq:bin-conditioned-0}
\Pr\bigl(\big|S^{(j+1)}_{\tau_{j+1}}\big| > A x \bigmid \cG_{j} \bigr)  &\geq u_\alpha (A+B) - \eps ,
\end{align}
where $u_\alpha (w) = r_\alpha s_\alpha (w-1)$,
with $r_\alpha$ and $s_\alpha$ given in equation~\eqref{eq:good_absorption_probability_away}
and~\eqref{eq:prob_large_overshoot}, which gives the expression in~\eqref{eq:u-alpha-def}.

Recall from equation~\eqref{eq:k_lost_particles} 
that $k_x(m,A)$ counts the number of particles in $\Cl_{\sigma_x + m} \setminus [-Ax,Ax]$. Now observe that  if $k_x(m,A)  < Bx$, then 
at time $\sigma_x +m$,  
for every $y$ with $|y|>(A+B)x$ we have at least one unoccupied site in $[y-\frac{Bx}{2},y+\frac{Bx}{2}]$, since the length of the interval $[y-\frac{Bx}{2},y+\frac{Bx}{2}]$ is $Bx$, and thus,  with the notation in~\eqref{notation-interval-empty}, we have $I(x,y,\sigma_x+m) \neq \emptyset$.
Hence, from~\eqref{eq:k-inc-prob} with the $j = \sigma_x +m$
case of equation~\eqref{eq:bin-conditioned-0}, 
\begin{align}
\label{eq:bin-conditioned}
    \Pr (\Delta_x (m,A) = 1  \mid \cG_{\sigma_x+m} )  \geq u_\alpha (A+B) - \eps , \text{ on } \{ k_x (m,A) < Bx \}.
\end{align}

Now for $\theta \in (0,B)$, let 
\begin{equation}\label{lambda_x-definition}
    \lambda_x: =  \inf\{m \in \ZP : k_x(m,A) > \theta x\}.
\end{equation}
Note that $\{k_x(x,A) \leq \theta x\} = \{\lambda_x >  x\}$.
For $0 <u < u_\alpha (A+B)$, 
let $M_m: = k_x(m\wedge \lambda_x,A) - u (m\wedge\lambda_x)$. Recall that $k_x(m,A)=k_x (0,A) + \sum_{j=0}^{m-1}\Delta_x (j,A)$. 
Since $\theta < B$, we can (and do) assume that $x$ is large enough so that $\theta x +1 < Bx$. 
Then observe that, for all $x$ sufficiently large, it follows from equation~\eqref{eq:bin-conditioned} that
\begin{eqnarray*}
\Exp \left(\Delta_x (m\wedge \lambda_x ,A) \mid \cG_{\sigma_x+m}\right)\ge u, \as,
\end{eqnarray*}
since for any $m < \lambda_x, \, k_x(m \wedge \lambda_x,A) \le \theta x +1 < Bx$ by the definition of $\lambda_x$.
Thus, for all $x$ sufficiently large, from the Doob  decomposition of $k_x(m,A)$ (see e.g.~Theorem 5.2.10 of~\cite{durrett}), it follows that 
$(M_m)_{m \in \ZP}$ is a submartingale adapted to $(\cG_{\sigma_x+m})_{m \in \ZP}$. 

From~\eqref{eq:bin-conditioned} and the fact that
$\Delta_x (m,A) \in \{0,1\}$
it follows that for $0 < \theta< u \wedge B$ 
\begin{align}
\label{eq:lambda_x-wait-x}
\Pr (\lambda_x > x \mid \cG_{\sigma_x}) &= \Pr (k_x(x,A) < \theta x, \, \lambda_x>x \mid \cG_{\sigma_x})
\nonumber\\
&= \Pr (k_x(x,A)-u x<(\theta -u)x, \, \lambda_x>x \mid \cG_{\sigma_x}) \nonumber\\
& = \Pr ( M_x < (\theta -u)x, \, \lambda_x>x \mid \cG_{\sigma_x}) \nonumber\\
        &\leq\Pr (M_x<(\theta -u)x \mid \cG_{\sigma_x}) \leq \exp(-(\theta -u)^2 x/2), 
    \end{align}
 where we use the Azuma--Hoeffding inequality for submartingales, 
Theorem~2.4.14
of~\cite{bluebook}.
This yields~\eqref{eq:lost-particles-bound}. \end{proof}

{Next we prove Proposition~\ref{prop:inner-radius-infinite-variance-lower-i},
making precise the idea outlined before the statement of Lemma~\ref{lem:d-l-estimate} about ``lost particles'' being generated in good proportion.}

\begin{proof}[Proof of Proposition~\ref{prop:inner-radius-infinite-variance-lower-i}]
First observe that from~\eqref{eq:sigma_limsup_infinite_variance_upper_bound} established in Proposition~\ref{prop:upper-bound_infinite-variance} above, for every $C > C_\alpha'>2$ as given in equation~\eqref{eq:C-alpha-dash}, a.s., for all but finitely many $x \in \N$,  it holds that 
 $\sigma_x \leq C x$.
 Thus, at time $Cx$, there are $2x$ of the walks that have occupied sites in $[-x,x]$,
 and at most $C-2x$ of the walks that are outside $[-x,x]$. 
It follows from definition of $k_x$ in equation~\eqref{eq:k_lost_particles} that, for every $A \geq 1$, a.s., for all but finitely many $x\in\N$,
$k_x (0, A) \leq (C-2) x$. In particular, we can choose $A>1$ and $B = C -2 >0$, to conclude that
$
k_x (0, A) \leq B x
$
for all but finitely many $x \in \N$. Then from~\eqref{eq:lost-particles-bound} we have that
for every $u \in (0,u_\alpha (A+B))$ and every $\theta \in (0, B \wedge u)$, $\sum_{x \in \N} \Pr ( k_x (x, A) \leq \theta x \mid \cG_{\sigma_x} ) < \infty$. Hence, by L\'evy's conditional Borel--Cantelli lemma~\cite[p.~131]{kall}, we conclude 
$k_x (x, A) > \theta x$ for all but finitely many $x \in \N$, a.s.

Now choose $A = (C+1)/2$.
If $\sigma_y \leq C y$
and $k_y(y,A) >  \theta y$,
then $\sigma_y + y \leq (C+1) y = 2 Ay$,
meaning that at time $2A y$ there are more than $\theta y$ sites outside $[-Ay, Ay]$. Therefore at time $2Ay$
there are  at least  $\theta y$ empty sites in $[-Ay,Ay]$. Those empty sites  must be filled before time $\sigma_{Ay}$ and consequently, $\sigma_{Ay} \geq 2Ay + \theta y$.
Taking $y = \lfloor x / A \rfloor$,  
we conclude that, a.s., 
\begin{equation}\label{theta-C-bound}
\liminf_{x \to \infty} \frac{\sigma_{x}}{x} \geq 2 + \frac{2 \theta}{C+1} ,
\end{equation}
where $C  > C_\alpha'$. The inequality \eqref{theta-C-bound} holds true for any positive $\theta  \leq B \wedge u_\alpha(A+B)$, with  $B = C-2$ and $A >1$. Therefore, we can choose $A$ such that  $A + B< (3/2)(C- 1)$ to conclude from $u_\alpha(A+ B) \geq u_{\alpha}((3/2)(C-1))$ that 
\[
\liminf_{x \to \infty} \frac{\sigma_{x}}{x} \geq 2 + \frac{(C-2) \wedge u_\alpha ( (3/2) (C-1) )}{C+1}, \as
\]
Since $C > C_\alpha'$ was arbitrary, this completes the proof.
\end{proof}

{Finally, we conclude this section with the proof of Theorem~\ref{thm:idla-heavy-tail},
which follows from the bounds in Propositions~\ref{prop:upper-bound_infinite-variance} 
and~\ref{prop:inner-radius-infinite-variance-lower-i} and the inversion relation~\eqref{eq:inversion}.}

\begin{proof}[Proof of Theorem~\ref{thm:idla-heavy-tail}]
First, with $C_\alpha' \in (2,\infty)$ as given in equation~\eqref{eq:C-alpha-dash}, we have from~\eqref{eq:sigma_limsup_infinite_variance_upper_bound} in Proposition~\ref{prop:upper-bound_infinite-variance} 
that $\limsup_{x \to \infty} \sigma_x/x \leq C'_\alpha$, a.s.
Together with the inversion relation~\eqref{eq:sigma-x-def}--\eqref{eq:inversion}, it follows immediately that
$\liminf_{m \to \infty} r_m /m \geq 1/C'_\alpha$, a.s., which yields the lower bound in~\eqref{eq:r-m-heavy}
with $c_\alpha := 1/C'_\alpha \in (0, 1/2)$.
The formula~\eqref{eq:c-alpha} for $c_\alpha$ follows directly from formula~\eqref{eq:C-alpha-dash} for $C'_\alpha$.

In the other direction, we have from~\eqref{eq:sigma_liminf_infinite_variance-prop_lower}  in  Proposition~\ref{prop:inner-radius-infinite-variance-lower-i} 
that
$\liminf_{x \to \infty} \sigma_x/x \geq C_\alpha''$, a.s., 
where $C_\alpha' \geq C_\alpha'' > 2$ is as defined in Proposition~\ref{prop:inner-radius-infinite-variance-lower-i}.
Again, inversion then yields $\limsup_{m \to \infty} r_m/m \leq c'_\alpha$, a.s., where $c'_\alpha := 1 / C''_\alpha$
satisfies $0 < c_\alpha \leq c'_\alpha < 1/2$, completing the proof of~\eqref{eq:r-m-heavy} and hence of Theorem~\ref{thm:idla-heavy-tail}.
\end{proof}


\begin{appendix}

\section{Eventual filling}
\label{sec:appendix}

This short appendix presents the proof of Proposition~\ref{prop:fill-finite-set}. As in \S\ref{sec:inner-radius-infinite-variance-lower-bound}, we write $\cG_m = \sigma( S^{(1)}, \ldots, S^{(m)})$.

\begin{proof}[Proof of Proposition~\ref{prop:fill-finite-set}]
It suffices to prove that for every $A \subseteq \Z$ finite, we have $A \subseteq \Cl_\infty$, a.s. Fix such an $A$.
Let $z \in A$. Then, by the irreducibility property~\eqref{hyp:irreducible}, there exist
$p_z >0$, 
$n_z \in \ZP$, and $x_{z,1}, \ldots, x_{z,n_z-1} \in \Z \setminus \{0,z\}$ such that
\[ \Pr (S_1 =x_{z,1}, \ldots, S_{n_z-1} = x_{z,n_z-1}, S_n = z) =p_z > 0.\]
With probability $p_z^{n_z}$, a sequence of $n_z$ successive random walkers will follow the path $x_{z,1}, \ldots, x_{z,n_z-1}, z$ for their first~$n_z$ steps; on this event,
at least one of the walkers will reach~$z$ before terminating, and so $z \in \Cl_\infty$.
Set $p_A:= \prod_{z \in A} p_z^{n_z} > 0$ and  $n_A := \sum_{z \in A} n_z < \infty$.
Given $\Cl_m$, iterating
the above argument shows that (regardless of the existing configuration),
with probability $p_A$, the sequence of random walkers $m+1, m+2, \ldots, m+ n_A$
executes an event such that $A \subseteq \Cl_{m+n_A}$; that is,
\[ \Pr ( A \subseteq \Cl_{\infty} \mid \cG_m ) \geq \Pr ( A \subseteq \Cl_{m +n_A} \mid \cG_m ) \geq p_A, \as \]
Denoting $\cG_\infty := \sigma ( \cup_{m \in \ZP} \cG_n )$, it follows from L\'evy's zero--one law
(see e.g.~Theorem~5.5.8 of~\cite{durrett})
that
\[ \Pr ( A \subseteq \Cl_\infty \mid \cG_\infty ) = \lim_{m \to \infty} \Pr (A \subseteq \Cl_\infty \mid \cG_m ) \geq p_A, \as \]
But since $\{ A \subseteq \Cl_\infty \} \in \cG_\infty$, this means that
$\1 {A \subseteq \Cl_\infty } \geq p_A$, a.s., for deterministic $p_A >0$, from which
it must hold that $\Pr ( A \subseteq \Cl_\infty ) = 1$.
\end{proof}

\section{Equicontinuity}
\label{sec:appendix-equicontinuity}

 The aim of this section is to prove Lemma~\ref{lem:kesten-equicontinuity}. The first step is
the following lemma, which is essentially provided already by Kesten.
Recall that the local hitting property, Lemma~\ref{lem:recurrence-hitting}, expresses the fact that a recurrent random walk will very likely hit a point at finite distance from its starting point before going far away; Lemma~\ref{lemma:hitting_immediate_neighbours_whp} extends this statement from points at a finite distance to points at a distance allowed to grow slowly with the size of the interval.
In this section of the appendix, we use the same notation as \S\ref{sec:kesten}; we recall in particular the notation  $T_t$ and $\eta_A$ from~\eqref{eq:T-eta-S}.

\begin{lemma}[Kesten 1961~\cite{kesten1961b}]
\label{lemma:hitting_immediate_neighbours_whp}
 Suppose that~\eqref{ass:irreducible} holds. For $1 < \alpha \leq 2$, suppose in addition that (if $\alpha = 2$) $\Exp ( X^2 ) < \infty$ and $\Exp X =0$, or (if $1<\alpha<2$) that~\eqref{ass:stable-alpha} holds. 
Then the following hold.
\begin{thmenumi}[label=(\roman*)]
\item\label{lemma:hitting_immediate_neighbours_whp-i} For any $\eps>0$, there exist $\delta >0$ and $N_0 \in \N$, such that, for all $N \geq N_0$ and every $k \in \Z$ with $| k | \le \delta  N$,
\begin{equation}
\label{Eq:Kesten_hit_immediate_points}
\bP_k \left( T_0 <  \eta_{(-\infty, N]} \right) \ge 1-\eps.
\end{equation}
\item\label{lemma:hitting_immediate_neighbours_whp-ii} For any $\eps >0$ and $c>0$, 
there exist $\delta >0$ and $N_0 \in \N$, such that, for all $N \geq N_0$ and every $k \in \Z$ with $|k| \le \delta N$,
\begin{equation}
\label{Eq:Kesten_hit_neighbors_with_2_sided_boundary}
\bP_k \left( T_0 < \eta_{[-cN, N]} \right) \ge 1-\eps.
\end{equation}
\end{thmenumi}
\end{lemma}
\begin{proof}
Part~\ref{lemma:hitting_immediate_neighbours_whp-i} is already available as Lemma~3 in~\cite{kesten1961b} for $1 < \alpha <2$. For $\alpha=2$, it is available as Lemma~3 in~\cite{kesten1961b} under the additional assumption of symmetric increment distribution, i.e.,  $X \eqd -X$. If we remove this assumption, but instead assume that $\Exp X=0$ and $\Exp(X^2)< \infty$, then it is easy to show that from Taylor expansion and dominated convergence that
\begin{equation*}
\label{Eq:cong_of_cht_fn}
\lim_{t \to 0} \frac{1}{t^2} \left(1 - \Exp \left[\re^{itX}\right]\right) = \Exp(X^2) ;  
\end{equation*}
observe also that $\Exp X =0$ and~\eqref{ass:irreducible} together imply that $\Exp (X^2) > 0$.
This observation is the key ingredient in the proof of Lemma 3 in~\cite{kesten1961b} (see equation (2.3) which leads to (2.31)  in~\cite{kesten1961b}), and hence an exact verbatim proof as in Lemma 3 in~\cite{kesten1961b} yields~\eqref{Eq:Kesten_hit_immediate_points}.

Part~\ref{lemma:hitting_immediate_neighbours_whp-ii} can be deduced from part~\ref{lemma:hitting_immediate_neighbours_whp-i}, as is also indicated in the proof of Lemma 4 in~\cite{kesten1961b}. Indeed,  observe that for any $c>0$, we have
\[ \{ T_0 < \eta_{(-\infty, N]} \} = \bigcup_{m \in \N} \{ T_0 < \eta_{[-cm, N]} \} , \]
where, for $m \in \N$,   the monotonicity property
$ \{ T_0 < \eta_{[-cm,N]} \} \subseteq \{ T_0 < \eta_{[-c(m+1),N]}\}$ is satisfied. Hence, by continuity of monotone limits, 
\[ \lim_{m \to \infty}\bP_k \left( T_0 < \eta_{[-cm, N]}  \right) 
=\bP_k \left( T_0 <  \eta_{(-\infty, N]} \right).
\]
This, together with \eqref{Eq:Kesten_hit_immediate_points} completes the proof of~\eqref{Eq:Kesten_hit_neighbors_with_2_sided_boundary}.
\end{proof}

We are now ready to present the proof of Lemma~\ref{lem:kesten-equicontinuity}.
\begin{proof}[Proof of Lemma~\ref{lem:kesten-equicontinuity}]
The proof for~\ref{lem:kesten-equicontinuity-ii} is already available in Lemma 4 in \cite{kesten1961b}, so we only prove part~\ref{lem:kesten-equicontinuity-i} here, which is also very similar. 

Suppose that $S_0 =k$, and observe that, for  $0<k <N$, 
\begin{equation}
\{ T_{-1} < \eta_{(-\infty, N]} \} \subseteq 
\{ S_{\eta_{[0,N]}} < 0 \}.
\end{equation}
This shows that for any $| y | < \delta(\eps)$, 
\begin{equation}
p_{\alpha, N} ( y) \ge 1-\eps.
\end{equation}
Therefore, for any $0 \le y_1, y_2 \leq \delta(\eps)$, since $0 \le p_{\alpha, N} ( y) \le 1$
\begin{equation}
p_{\alpha, N} ( y_1) \ge \left(1-\eps\right)p_{\alpha, N} ( y_2),
\end{equation}
which proves equicontinuity for $| y | < \delta(\eps)$. Hence, it is enough to show equicontinuity for any $y > \delta(\eps)$, the proof of which is similar to that of part~\ref{lemma:hitting_immediate_neighbours_whp-ii}. We provide the details for the purpose of completeness.

Let $k_1, k_2 \ge \delta(\eps) N$, 
it follows from the Markov property and translation invariance, that 
\begin{equation}
\label{Eq:Prob_k_1_greater_k_2}
p_{\alpha, N} \Bigl( \frac{k_1}{N} \Bigr) \ge \bP_{k_1} \left( S_n = k_2, \text{ for some } 1 \le n < \eta_{[\frac{\delta(\eps)}{2} N, N]}\right)p_{\alpha, N} \Bigl( \frac{k_2}{N} \Bigr).
\end{equation}
It follows from translation invariance and \eqref{Eq:Kesten_hit_neighbors_with_2_sided_boundary}, that for any given $\eps>0$, there exits $\delta_1(\eps)$, such that
\begin{equation}
\bP_{k_1} \left( S_n = k_2, \text{ for some } 1 \le n < \eta_{[\frac{\delta(\eps)}{2} N, N]}\right) \ge 1-\eps,
\end{equation}
 whenever $k_1, k_2 \le N(1-\eps) $, and $| k_1 -k_2 | \leq N \delta_1(\eps)$. Thus we can choose $\delta_2(\eps)>0$, such that, for $0 \le y_1, y_2< 1-\eps$ and $| y_1 -y_2 | <\delta_2(\eps )$, such that,
 \begin{equation}
 \label{Eq:tilde_py_1_more}
 p_{\alpha, N}(y_1) \ge (1-\eps) p_{\alpha, N}(y_2).
 \end{equation}
 Since, $y_1$ and $y_2$ are arbitrary and $0 \le p_{\alpha, N}(y) \le 1$, this proves~\ref{lem:kesten-equicontinuity-i}.
\end{proof}

\end{appendix}

\section*{Acknowledgements}
\addcontentsline{toc}{section}{Acknowledgements}
CdC and AW were supported by EPSRC grant EP/W00657X/1. 
The majority of this work was
done when the first author was employed by Durham University. The 
authors are grateful to 
two anonymous referees for their valuable comments, corrections, and suggestions. We
thank 
Wioletta Ruszel for suggesting the model to us, and M.\ Bal\'asz,  E.\ Crane, and J.\ Jay for valuable discussions, and 
for making available the draft version of~\cite{Balasz_etal2024}. We particularly thank E.\ Crane for bringing the important reference~\cite{blachere} to our attention.
 Some of this work was carried
 out during the programme ``Stochastic systems for anomalous diffusion'' (July--December 2024) hosted by the  Isaac Newton Institute, under EPSRC grant EP/Z000580/1.


\begin{thebibliography}{99}

\bibitem{as}
M.\ Abramowitz and I.A.\ Stegun (Eds.),  
\emph{Handbook of Mathematical Functions}, National Bureau of Standards,
Applied Mathematics Series, {\bf 55}. U.S.\ Government Printing
Office, Washington D.C., 1965.

\bibitem{amir1}
G.\ Amir, O.\ Angel, I.\ Benjamini, and G.\ Kozma,
One-dimensional long-range diffusion-limited aggregation~I.
\emph{Ann.\ Probab.}\ {\bf 44} (2016) 3546--3579.

\bibitem{amir2}
G.\ Amir, O.\ Angel, and G.\ Kozma,
One-dimensional long-range diffusion limited aggregation II: The transient case.
\emph{Ann.\ Appl.\ Probab.}\ {\bf 27} (2017) 1886--1922.

\bibitem{amir3}
G.\ Amir,
One-dimensional long-range diffusion-limited aggregation III: The limit aggregate. 
\emph{Ann.\ Inst.\ H.\ Poincar\'e Probab.\ Statist.}\ {\bf 53} (2017) 1513--1527.

\bibitem{asmussen}
S.\ Asmussen,
\emph{Applied Probability and Queues}. 
2nd ed., Springer, New York, 2003.

\bibitem{ag1}
A.\ Asselah and A.\ Gaudilli\`ere,
From logarithmic to subdiffusive polynomial fluctuations for internal DLA and related growth models. 
\emph{Ann.\ Probab.}\ {\bf 41} (2013) 1115--1159.

\bibitem{ag2}
A.\ Asselah and A.\ Gaudilli\`ere,
Sub-logarithmic fluctuations for internal DLA. 
\emph{Ann.\ Probab.}\ {\bf 41} (2013) 1160--1179.

\bibitem{ast}
A.\ Asselah, V.\ Silvestri, and L.\ Taggi, 
Internal diffusion limited aggregation with critical branching random walks.
Preprint (2025)
\href{https://arxiv.org/abs/2510.13733}{arXiv:2510.13733}.

\bibitem{bpp}
M.T.\ Barlow, R.\ Pemantle, and E.A.\ Perkins,
Diffusion-limited aggregation on a tree.
\emph{Probab.\ Theory Relat.\ Fields} {\bf 107} (1997) 1--60.

\bibitem{Balasz_etal2024}
M.\ Bal\'{a}zs, E.\ Crane, J.\ Jay, and D.\ Thacker,
Internal diffusion-limited aggregation with totally asymmetric long-range jumps, and its connection to ASEP blocking measures and dynamical Mallows permutations of the integers (in preparation).

\bibitem{ben-arous}
G.\ Ben~Arous, J.\ Quastel, and A.F.\ Ram\'irez. 
Internal DLA in a random environment.
\emph{Ann.\ Inst.\ H.\ Poincar\'e Probab.\ Statist.}\
{\bf 39} (2003) 301--324.

\bibitem{bdckl}
I.\ Benjamini, H.\ Duminil-Copin, G.\ Kozma, and C.\ Lucas,
Internal diffusion-limited aggregation with uniform starting points.
\emph{Ann.\ Inst.\ H.\ Poincar\'e Probab.\ Statist.}\
{\bf 56} (2020) 391--404.

\bibitem{Be_Ka_Pr_2014}
N.\ Berger, J.J.\ Kagan, and E.B.\ Procaccia,
Stretched IDLA.
\emph{ALEA.}\ {\bf 11} (2014) 471--481.

\bibitem{bgt}
N.H.\ Bingham, C.M.\ Goldie and J.L.\ Teugels,
\emph{Regular Variation}.
Cambridge University Press, Cambridge, 1989.

\bibitem{blachere}
S.\ Blach\`ere,
Agr\'egation limit\'ee par diffusion interne sur {$\mathbb{Z}^d$}.
\emph{Ann.\ Inst.\ H.\ Poincar\'e Probab.\ Statist.}\
  {\bf 38}, (2002) 613--648.
  
\bibitem{chang}
J.T.\ Chang,
Inequalities for the overshoot.
\emph{Ann.\ Appl.\ Probab.}\ {\bf 4} (1994) 1223--1233.

\bibitem{chsht}
J.P.\ Chen, A.\ Huss, E.\ Sava-Huss and A.\ Teplyaev, 
Internal DLA on Sierpinski gasket graphs.
\emph{Analysis and geometry on graphs and manifolds}, pp.~126--155, 
Cambridge University Press, Cambridge, 2020.

\bibitem{Ch_Ki_So2010}
Z.-Q.\ Chen, P.\ Kim and R.\ Song,
Heat kernel estimates for the Dirichlet fractional Laplacian.
\emph{J.\ Eur.\ Math.\ Soc.}\ {\bf 12} (2010) 1307--1329.

\bibitem{Che_Cou_Rou_2023}
N.\ Chenavier, D.\ Coupier and A.\ Rousselle,
The bi-dimensional directed IDLA forest,
\emph{Ann.\ Appl.\ Probab.}\ {\bf 33} (2023) 2247--2290.

\bibitem{df}
P.\ Diaconis and W.\ Fulton,
A growth model, a game, an algebra, Lagrange inversion, and characteristic classes. 
\emph{Rend.\ Semin.\ Mat.\ Univ.\ Politec.\ Torino} {\bf 49} (1991) 95--119.

 \bibitem{doney-ladder}
 R.A.\ Doney,
 Moments of ladder heights in random walks.
 \emph{J.\ Appl.\ Probab.}\ {\bf 17} (1980) 248--252.

 \bibitem{doney}
 R.A.\ Doney, 
 One-sided local large deviation and renewal theorems in the case of infinite mean.
 \emph{Probab.\ Theory Relat.\ Fields} {\bf 107} (1997) 451--465.

\bibitem{DrMota_2020}
M.\ Drmota, M.\ Fuchs, K.-H.\ Hwang and R.\ Neininger, 
Node profiles of symmetric digital search trees: Concentration properties.
\emph{Random Struct.\ Algorithms} {\bf 58} (2020) 430--467.

 \bibitem{durrett}
 R.\ Durrett, 
 \emph{Probability: Theory and Examples.} 4th ed., Cambridge University Press, Cambridge, 2010.

 \bibitem{feller2}
 W.\ Feller, 
 \emph{An Introduction to Probability Theory and Its Applications, Vol.\ II.}
 2nd ed., John Wiley \& Sons, Inc., New York, 1971.

\bibitem{freedman}
D.A.\ Freedman,
Bernard Friedman's urn.
\emph{Ann.\ Math.\ Statist.}\ {\bf 36} (1965) 956--970.

\bibitem{fhksh}
U.\ Freiberg, N.\ Heizmann, R.\ Kaiser and E.\ Sava-Huss, 
Internal aggregation models with multiple sources and obstacle problems on Sierpiński gaskets. \emph{J.\ Fractal Geom.}\ {\bf 11} (2024) 111--160.

\bibitem{gut-srw}
A.\ Gut, 
\emph{Stopped Random Walks}. 
2nd ed., Springer, New York, 2009.

\bibitem{il}
I.A.\ Ibragimov and Ju.V.\ Linnik,
\emph{Independent and Stationary Sequences of Random Variables.}
Wolters-Noordhoff, Groningen, 1971.

\bibitem{janson}
S.\ Janson,
Renewal theory for $m$-dependent variables.
\emph{Ann.\ Probab.}\ {\bf 11} (1983) 558--568.

\bibitem{jls1}
D.\ Jerison, L.\ Levine, and S.\ Sheffield, 
Logarithmic fluctuations for internal DLA. 
\emph{J.\ Amer.\ Math.\ Soc.}\ {\bf 25} (2012) 271--301. 

\bibitem{jls2}
D.\ Jerison, L.\ Levine, and S.\ Sheffield, 
Internal DLA in higher dimensions. 
\emph{Electron.\ J.\ Probab.}\ {\bf 18} (2013) no.~14.

\bibitem{jls3}
D.\ Jerison, L.\ Levine, and S.\ Sheffield, 
Internal DLA and the Gaussian free field. 
\emph{Duke Math.\ J.}\ {\bf 163} (2014) 267--308.

\bibitem{kall}
O.\ Kallenberg, 
\emph{Foundations of Modern Probability}.
2nd edition, Springer, New York,
2002.

\bibitem{kesten1961a}
H.\ Kesten, 
On a theorem of Spitzer and Stone and random walks with absorbing barriers.
\emph{Illinois J.\ Math.}\ {\bf 5} (1961) 246--266.

\bibitem{kesten1961b}
H.\ Kesten, 
Random walks with absorbing barriers and Toeplitz forms.
\emph{Illinois J.\ Math.}\ {\bf 5} (1961) 267--290.

 \bibitem{Kesten1987}
 H.\ Kesten,
 How long are the arms in DLA? 
 \emph{J.\ Phys.\ A: Math.\ Gen.}\ {\bf 20} (1987) L29--L33.

 \bibitem{lai}
 T.L.\ Lai, 
 Asymptotic moments of random walks with applications to ladder variables and renewal theory.
 \emph{Ann.\ Probab.}\ {\bf 4} (1976) 51--66. 

\bibitem{lawler}
G.F.\ Lawler, 
Subdiffusive fluctuations for internal diffusion limited aggregation.
\emph{Ann.\ Probab.}\ {\bf 23} (1995) 71--86.

\bibitem{lbg}
G.F.\ Lawler, M.\ Bramson, and D.\ Griffeath,
Internal diffusion limited aggregation.
\emph{Ann.\ Probab.}\ {\bf 20} (1992) 2117--2140.

\bibitem{llbook}
 G.F.\ Lawler and V.\ Limic,
\emph{Random Walk: A Modern Introduction.}
Cambridge University Press, 
Cambridge, 
2010.

\bibitem{lorden}
G.\ Lorden,
On the excess over the boundary.
\emph{Ann.\ Math.\ Statist.}\ {\bf 41} (1970) 520--527.

\bibitem{lotov}
V.I.\ Lotov,
 Bounds for the probability to leave the interval.
\emph{Statist.\ Probab.\ Letters} {\bf 145} (2019) 141--146.
 
\bibitem{lucas}
C.\ Lucas, 
The limiting shape for drifted internal diffusion limited aggregation is a true heat ball. 
\emph{Probab.\ Theory Relat.\ Fields} {\bf 159} (2014) 197--235.

\bibitem{md}
P.\ Meakin and J.M.\ Deutch,
The formation of surfaces by diffusion limited annihilation.
\emph{J.\ Chem.\ Phys.}\ {\bf 85} (1986) 2320--2325.

\bibitem{bluebook}
M.\ Menshikov, S.\ Popov, and A.\ Wade,
\emph{Non-homogeneous Random Walks}.  
Cambridge University Press, Cambridge, 2016.

\bibitem{penrose}
M.\ Penrose,
\emph{Random Geometric Graphs}.
Oxford University Press, Oxford,
2003.

\bibitem{rs}
O.\ Raimond and B.\ Schapira, 
Internal DLA generated by cookie random walks on $\mathbb{Z}$.
\emph{Electron.\ Commun.\ Probab.}\ {\bf 16} (2011) 482--490.

\bibitem{rogozin}
B.A.\ Rogozin, 
On the distribution of the first ladder moment and height and fluctuations of a random walk. 
\emph{Theory Probab.\ Appl.}\ {\bf 16} (1971) 575--595.

\bibitem{Sava-Huss2021}
E.\ Sava-Huss,
From fractals in external DLA to internal DLA on fractals. pp.~273--298 in
\emph{Fractal Geometry and Stochastics VI}, Progress in Probability vol.~76, Springer,  2021.

\bibitem{silvestri}
V.\ Silvestri,
Internal DLA on cylinder graphs: fluctuations and mixing.
\emph{Electron.\ Commun.\ Probab.}\ {\bf 25} (2020) 1--14.

\bibitem{spitzer}
F.\ Spitzer,
A Tauberian theorem and its probability interpretation.
\emph{Trans.\ Amer.\ Math.\ Soc.}\ {\bf 94} (1960) 150--169.

\bibitem{ws}
T.A.\ Witten Jr.~and L.M.\ Sander, 
Diffusion-limited aggregation, a kinetic critical phenomenon.
\emph{Phys.\ Rev.\ Letters} {\bf 47} (1981) 1400.

\end{thebibliography}
\end{document}